\documentclass[a4paper, 11pt]{amsart}

\title[GV invariants of local curves]
{Degree two Gopakumar--Vafa invariants of local curves}
\author{Ben Davison and Naoki Koseki}
\date{}
\address{School of Mathematics and Maxwell Institute for Mathematical Sciences,
University of Edinburgh,
Edinburgh EH9 3FD,
UK}
\email{ben.davison@ed.ac.uk}
\address{The University of Liverpool, Mathematical Sciences Building, Liverpool, L69 7ZL, UK.}
\email{koseki@liverpool.ac.uk}

\makeatletter
 
  \@addtoreset{equation}{section}
 \makeatother

\usepackage{amsmath, amssymb, amsthm, amscd, comment, mathtools, color}
\usepackage{todonotes}
\usepackage[frame,cmtip,curve,arrow,matrix,line,graph]{xy}
\usepackage[mathscr]{euscript}
\usepackage{mathrsfs}
\usepackage{tikz}
\usepackage{tikz-cd}
\usetikzlibrary{intersections, calc}

\theoremstyle{plain}
\newtheorem{thm}{Theorem}[section]
\newtheorem{prop}[thm]{Proposition}
\newtheorem{def-prop}[thm]{Definition-Proposition}
\newtheorem{lem}[thm]{Lemma}
\newtheorem{cor}[thm]{Corollary}
\newtheorem*{thm*}{Theorem}

\theoremstyle{definition}
\newtheorem{defin}[thm]{Definition}
\newtheorem{conj}[thm]{Conjecture}

\newtheorem*{ACK}{Acknowledgements}

\theoremstyle{remark}
\newtheorem{rmk}[thm]{Remark}

\newtheorem{ex}[thm]{Example}

\DeclareMathOperator{\ch}{ch}
\DeclareMathOperator{\td}{td}
\DeclareMathOperator{\rk}{rk}

\DeclareMathOperator{\id}{id}
\newcommand{\dR}{\mathbf{R}}

\newcommand{\bP}{\mathbb{P}}
\newcommand{\bC}{\mathbb{C}}

\newcommand{\bR}{\mathbb{R}}
\newcommand{\bQ}{\mathbb{Q}}
\newcommand{\bZ}{\mathbb{Z}}

\newcommand{\mcC}{\mathcal{C}}

\newcommand{\mcF}{\mathcal{F}}

\newcommand{\mcH}{\mathcal{H}}

\newcommand{\mcM}{\mathcal{M}}

\newcommand{\mcO}{\mathcal{O}}

\newcommand{\mcV}{\mathcal{V}}

\DeclareMathOperator{\Hom}{Hom}
\DeclareMathOperator{\Tot}{Tot}

\DeclareMathOperator{\Pic}{Pic}

\DeclareMathOperator{\Ext}{Ext}

\DeclareMathOperator{\Hilb}{Hilb}
\DeclareMathOperator{\Sym}{Sym}
\DeclareMathOperator{\IC}{IC}
\DeclareMathOperator{\GW}{GW}
\DeclareMathOperator{\AJ}{AJ}

\DeclareMathOperator{\PGL}{PGL}

\DeclareMathOperator{\Chow}{Chow}
\DeclareMathOperator{\Perv}{Perv}

\DeclareMathOperator{\pt}{pt}
\DeclareMathOperator{\Gr}{Gr}

\DeclareMathOperator{\tr}{tr}

\DeclareMathOperator{\mH}{H}

\DeclareMathOperator{\Crit}{Crit}

\DeclareMathOperator{\sco}{\!\colon\!\!\!}

\begin{document}
\maketitle

\begin{abstract}
We investigate the Gopakumar--Vafa (GV) theory of local curves, namely, the total spaces of rank two vector bundles with canonical determinant on smooth projective curves. 
Under a certain genericity condition on the rank two bundles, we propose a general mechanism to compute the degree two GV invariants of local curves. 
In particular, we determine all the degree two GV invariants when the base curve has genus two. Combined with previous work by Bryan and Pandharipande, we obtain the GV/GW correspondence in this case. 
When the base curve has genus greater than two, we calculate GV invariants for some extremal genera, providing evidence for the GV/GW conjecture for curves of higher genus.
\end{abstract}

\setcounter{tocdepth}{1}
\tableofcontents

\section{Introduction}
\subsection{Motivation and results}
Throughout the paper, we work over the complex number field $\bC$. 
We investigate local enumerative invariants of a Calabi--Yau (CY) $3$-fold around a smooth curve. 
A model for this enumerative geometry is provided by \textit{local curves}. 
Here, a local curve is a CY $3$-fold of the form $\Tot_C(N)$, where $C$ is a smooth projective curve and $N$ is a rank two vector bundle on $C$ with $\det(N) \cong \omega_C$. 
Local curves are a fundamental object of study in enumerative geometry 
\cite{bp01, bp06, bp08, mon22, op10}. 
In particular, they have played a key role in the recent proof of the Gromov--Witten (GW)/Donaldson--Thomas (DT) correspondence for \textit{any} smooth projective CY $3$-fold \cite{par24}. 
Pardon reduced the problem to the case of (equivariant) local curves, and 
the GW/DT correspondence for the latter has been proved by directly computing the GW and DT invariants of local curves \cite{bp08, op10}. 
Moreover, if we impose a certain rigidity condition on the bundle $N$, 
the local GW invariants agree with 
those on a projective CY $3$-fold with curve class $r[C]$, where $C \subset X$ has normal bundle $N$ \cite{bp01, bp05}. 

In this paper we study another curve counting theory on a CY $3$-fold, that is \textit{Gopakumar--Vafa (GV) theory}. 
GV theory was originally introduced in physics \cite{GV98}, as a way to count curves in CY $3$-folds. 
Recently, Maulik--Toda \cite{mt18} proposed a mathematically rigorous definition of GV invariants, building on the previous works by Hosono--Saito--Takahashi \cite{hst01} and Kiem--Li \cite{kl16}. 
See the next subsection for the review of Maulik--Toda's work. 
GV theory is also expected to be equivalent to GW theory (the \textit{GV/GW correspondence conjecture}): 
\begin{conj} \label{conj:GVGW}
Let $X$ be a smooth projective CY $3$-fold. 
For a curve class $\beta \in H_2(X, \bZ)$ and a non-negative integer $g \in \bZ_{\geq 0}$, 
let $\GW_{g, \beta} \in \bQ$ 
(resp. $n_{g, \beta} \in \bZ$) 
denote the GW (resp. GV) invariant of $X$. 
Then
\begin{equation} \label{eq:GVGW}
\sum_{g, \beta}\GW_{g, \beta}\lambda^{2g-2}t^\beta
=\sum_{g, \beta, k \geq 0} \frac{n_{g, \beta}}{k}\left(
2\sin\left(\frac{k\lambda}{2}
\right) 
\right)^{2g-2}t^{k\beta}. 
\end{equation}
\end{conj}
If $X$ is a generic local curve, and we set $\beta=r[C]$, then in \cite{mt18} the GV/PT correspondence (and thus the GV/DT correspondence, via \cite{BriDTPT,TodaDTPT}) is proved for $r=1$ or $g(C)\leq 1$ \cite{hst01, mt18}.  
The number $r$ is called the \textit{degree} of the invariant.  As a first step toward a proof of the GV/GW correspondence conjecture, 
we investigate the \textit{degree two} GV invariants of a local curve in this paper. 
The following are our main results: 
\begin{thm} \label{thm:introgenustwo}
Let $C$ be a smooth projective curve of genus two, 
$L \in \Pic^3(C)$ a generic line bundle. 
Let $[N] \in \bP \Ext^1(L, L^{-1} \otimes \omega_C)$ 
be a generic rank two vector bundle on $C$, and let $X=\Tot_C(N)$. 
Then we have 
\[
n_{g, 2[C]}(X)=\begin{cases}
8 & (g=3), \\
-2 & (g=2), \\
0 & (\text{otherwise}). 
\end{cases}
\]
In particular, Conjecture \ref{conj:GVGW} holds in this case. 
\end{thm}

\begin{thm} \label{thm:introhigher}
Let $C$ be a smooth projective curve of arbitrary genus, and let
$L \in \Pic^{2g(C)-1}(C)$ be a generic line bundle. 
Let $[N] \in \bP \Ext^1(L, L^{-1} \otimes \omega_C)$ 
be a generic rank two vector bundle on $C$, and let $X=\Tot_C(N)$. 
Then
\[
n_{g, 2[C]}(X)=\begin{cases}
0 & (g \geq 4g(C)-2), \\
-2^{2g(C)-3} & (g=g(C)), \\
0 & (g \leq g(C)-1). 
\end{cases}
\]
In particular, the GV/GW correspondence holds for the above range of $g$. 
\end{thm}

\subsection{GV invariants of a local curve 
and strategy of the proofs}
First, we review the proposed definition of Maulik and Toda. 
Let $X$ be a smooth quasi-projective CY $3$-fold. We fix a curve class $\beta \in H_2(X, \bZ)$. 
Let $M_X(\beta, 1)$ be the moduli space of slope stable (with respect to a fixed ample divisor) one-dimensional sheaves $E$ on $X$ with $[E]=\beta$ and $\chi(E)=1$. 
We have the Hilbert--Chow morphism: 
\[
\pi \colon M_X(\beta, \chi) \to \Chow_\beta(X). 
\]
Then Maulik and Toda \cite{mt18} defined the invariants $n_{g, \beta}$ via the following identity: 
\begin{equation}
\label{eq:defGVintro}
\sum_{i \in \bZ}\chi({}^p\mcH^i(\dR \pi_* \phi_{M}))q^i=
\sum_{g \geq 0}n_{g, \beta}\left(q^{\frac{1}{2}}+q^{-\frac{1}{2}}
\right)^{2g}, 
\end{equation}
where $\phi_M$ on the left hand side 
denotes \textit{Joyce's perverse sheaf} 
on $M_X(\beta, 1)$. 
Here, and throughout, if $\mathcal{P}$ is a constructible complex of sheaves on a space $T$, we define $\chi(\mathcal{P})\coloneqq \sum_{i\in\mathbb{Z}}(-1)^i\dim (\mH^i(T,\mathcal{P}))$.
One of the main difficulties in this approach to GV theory is the complicated gluing procedure of the construction of $\phi_M$: 
it is defined as a gluing of locally defined perverse sheaves of vanishing cycles. 

Let us focus on the case of a local curve $X=\Tot_C(N)$. 
In this case, the curve class $\beta$ can be written as $\beta=r[C]$ for some integer $r >0$, where $[C]$ is the class of the zero section. 
We define $M_N(r, 1) \coloneqq M_X(r[C], 1)$. 
In this case, Kinjo--Masuda's result \cite{km24} simplifies the situation substantially. 
They proved that $M_N(r, 1)$ can be \textit{globally} written as a critical locus (see also \cite{kk23}): 
\[
M_N(r, 1) \cong \{dj=0\} \subset 
M_L(r, 1) \xrightarrow{j} \bC, 
\]
where $L$ is a line bundle of $\deg(L) \gg 0$, $M_L(r, 1)$ denotes the moduli space of slope stable sheaves on $\Tot_C(L)$, 
and the function $j$ is constructed as follows: the Hilbert--Chow morphism is identified with the \textit{Hitchin fibration} 
\[
h \colon M_L(r, 1) \to 
\tilde{B}_L=\bigoplus_{i=1}^r \mH^0(C, L^{\otimes i}). 
\]
Then $j$ is the composition of $h$ and a linear function $\tilde{f} \colon \tilde{B}_L \to \bC$. 
Moreover, we have $\phi_M \cong \phi_j(\IC_{M_L(r, 1)})$, and since $h$ is projective,
the left hand side of (\ref{eq:defGVintro}) is equal to 
\[
\sum_{i \in \bZ} \chi({^p\mcH^i(\phi_{\tilde{f}}(\dR h_* \IC_{M_L(r, 1)}))})q^i. 
\]
Here and throughout the paper, if $Y$ is a smooth equi-dimensional variety we denote by $\IC_Y\coloneqq \bQ_Y[\dim(Y)]$ the constant perverse sheaf on $Y$.  Using the equality above, and the distinguished triangle relating vanishing and nearby cycles, we will divide 
the computations of $n_{g, r[C]}(X)$ into the following two parts 
(see Proposition \ref{prop:GV=L+U}): 
\begin{enumerate}
\item GV type invariants of
$h \colon M_L(r, 1) \to \tilde{B}_L$, 
\item GV type invariants of 
$h|_{U_\varepsilon} \colon 
h^{-1}(U_\varepsilon) \to U_\varepsilon$, 
where $U_{\varepsilon}\coloneqq \{\tilde{f}=\epsilon\} \subset \tilde{B}_L$ ($0 < \varepsilon \ll 1$). 
\end{enumerate}
When $r=2$, we will give a general strategy to compute both (1) and (2). 
For (1), we will use the explicit description of $\mH^*(M_L(2, 1), \bQ)$ by generators and relations \cite{dCHM12, ht03, ht04}. 

For (2), we will calculate the fibrewise contributions and then integrate. 
Namely, we have a stratification of $U_\varepsilon$ given by the types of singularities of the spectral curves. 
Since the nearby hyperplane $U_\varepsilon$ avoids the origin and we assume $r=2$, the spectral curves $\widetilde{C}$ over 
points in $U_\varepsilon$ are all \textit{reduced} (more precisely, we first need to remove the trace part $\mH^0(L)$). Hence, by \cite{ms13, my14}, 
the fibrewise contribution can be computed by the Euler characteristics of the Hilbert schemes $\Hilb^n(\widetilde{C})$. 
The latter is then computed by certain HOMFLY polynomials \cite{mau16, os12}. 
It remains to compute the Euler characteristic of each stratum $S \subset U_\varepsilon$.
Our main idea is to realise (a resolution of) the closure of $S$ as a degeneracy locus inside a product of copies of $C$. 
These varieties are closely related to generalized de {J}onqui\`eres divisors \cite{far23, ung21}. 

In principal, the above strategy enables us to compute the GV invariants $n_{g, 2[C]}(\Tot_C(N))$ for an arbitrary curve $C$. 
However, the difficulty of the calculations in each step grows rapidly when we increase the genus $g(C)$. The authors currently do not know how to resolve this issue. 

\begin{rmk}
To generalise our strategy to the case $r>2$, there are at least two difficulties: 
\begin{enumerate}
\item For the first part, we do not have an explicit description of $\mH^*(M_L(r, 1))$ for $r \geq 3$. Instead, we may first need to prove the GV/PT correspondence for $\Tot_C(L \oplus L^{-1} \otimes \omega_C)$, 
and then use the result of \cite{mon22}. 
\item For the second part, a recursion formula determining the HOMFLY polynomials becomes complicated for $r \geq 3$. 
However, the serious difficulty is that we cannot assume the spectral curves over $U_\varepsilon$ to be reduced. 
\end{enumerate}
\end{rmk}

\subsection{Relation to previous work}
The GV/PT (and hence GV/GW) correspondence 
was previously proved in the following cases: 
\begin{itemize}
\item $X=\Tot_C(N)$ with $g(C) \leq 1$, $N$ is super-rigid, and $\beta$ is arbitrary \cite{hst01}, 
\item $X=\Tot_C(\omega_C \oplus \mcO_C)$ and $\beta$ is arbitrary \cite{cdp14, hmms22, ms24}, 
\item $X=\Tot_S(\omega_S)$ and $\beta$ is arbitrary, where $S$ is a K3 surface \cite{my22}, 
\item $X=\Tot_S(\omega_S)$ and $\beta$ is an irreducible curve class, where $S$ is any smooth projective surface \cite{mt18}. 
\end{itemize}

For super-rigid rational and elliptic curves in the first case, the key fact proved in \cite{hst01} is that the moduli space $M_X(\beta, 1)$ is smooth and projective. 

For $X=\Tot_C(\omega_C \oplus \mcO_C)$, 
as observed in \cite{cdp14}, 
the GV/PT correspondence follows from the P=W conjecture, which has recently been proved \cite{hmms22, ms24}. 

Shen--Yin \cite{my22} investigated a compact analogue of the P=W phenomenon for Lagrangian fibrations. As an application, they established the GV/PT correspondence for a local K3 surface. 

For an arbitrary local surface $\Tot_S(\omega_S)$, 
Maulik and Toda \cite{mt18} used a similar idea as in this paper, namely, they described the moduli space $M_X(\beta, 1)$ as an explicit critical locus, locally over $\Chow_\beta(X)$, so that they can use the results of \cite{ms13, my14}.

\subsection{Plan of the paper}
The paper is organised as follows. 
In Section \ref{sec:pre}, we collect some preliminary results. In particular, we define the GV invariants for local curves and prove that they are computed by the contributions from twisted Higgs bundles and the nearby hyperplane (Proposition \ref{prop:GV=L+U}). 
In Section \ref{sec:genustwo}, we prove Theorems \ref{thm:introgenustwo} and \ref{thm:introhigher}. 
In Section \ref{sec:highergenus}, we prove Theorem \ref{thm:introhigher}. 
In Appendix \ref{sec:GW}, we recall some known results in GW theory that we use in the main text for verifying cases of the GV/GW correspondence. 

\begin{ACK}
We would like to thank 
Jim Bryan for introducing this problem to us. 
We would also like to thank Tasuki Kinjo, Davesh Maulik and Yukinobu Toda for related discussions.  We also thank the anonymous referee for many helpful suggestions and corrections.

B.D is supported by a Royal Society university research fellowship \newline [URF/R/221040].

N.K. is supported by the Engineering and Physical Sciences Research Council [EP/Y009959/1]. 
\end{ACK}

\section{Preliminaries} \label{sec:pre}
\subsection{Twisted Higgs bundles}
Let $C$ be a smooth projective curve. 

\begin{defin}
\begin{enumerate}
Let $\mcV$ be a vector bundle on $C$. 
\item A \textit{$\mcV$-twisted Higgs bundle} is a pair $(E, \phi)$ consisting of 
a vector bundle $E$ on $C$ and a homomorphism $\phi \colon E \to E \otimes \mcV$ 
such that $\phi \wedge \phi=0$. 
\item A $\mcV$-twisted Higgs bundle $(E, \phi)$ is \textit{$\mu$-stable} 
if the following condition holds: 
for any subsheaf $F \subset E$ with $\rk(F) < \rk(E)$ and $\phi(F) \subset F \otimes \mcV$, 
we have 
\[
\frac{\deg(F)}{\rk(F)} < \frac{\deg(E)}{\rk(E)}. 
\]
\end{enumerate}
\end{defin}

We denote by $\mcM_{\mcV}(r, d)$ (resp. $M_{\mcV}^{st}(r, d)$) 
the moduli stack of $\mcV$-twisted Higgs bundles 
(resp. the moduli space of $\mu$-stable $\mcV$-twisted Higgs bundles) 
of rank $r$ and degree $d$. 
\begin{thm}\cite[Theorem 5.6, Proposition 5.7]{km24}
\label{thm:km}
Let $N$ be a rank two vector bundle on $C$ with $\det(N) \cong \omega_C$. 
Choose an exact sequence 
\[
0 \to L^{-1} \otimes \omega_C \to N \to L \to 0 
\]
for some line bundle $L$ with $\deg(L) \geq 2g(C)-1$. 
Let $f \colon \mH^0(L^{\otimes 2}) \to \bC$ be a linear function 
corresponding to the class $[N] \in \Ext^1(L, L^{-1} \otimes \omega_C)$ 
via Serre duality: 
$Ext^1(L, L^{-1} \otimes \omega_C) \cong \mH^0(L^{\otimes 2})^\vee$. 
Then we have an isomorphism 
\begin{equation*} 
\mcM_N(r, d) \cong \Crit(j) \subset \mcM_L(r, d), 
\end{equation*}
where $j \colon \mcM_L(r, d) \to \bC$ is the function 
defined by $j((E, \phi))=
f(\tr(\phi^2))$. 
\end{thm}

\begin{cor} \label{cor:rank2dchart}
In the setting of Theorem \ref{thm:km}, 
assume further that $r=2$. 
Then we have 
\begin{equation} \label{eq:dchart}
M^{st}_N(2, d) \cong \Crit(j) \subset M^{st}_L(2, d). 
\end{equation}
\end{cor}
\begin{proof}
We need to prove that the inclusion 
$i \colon \mcM_N(2, d) \subset \mcM_L(2, d)$ 
preserves $\mu$-stability. 
Suppose for a contradiction that there exists $(E, \phi) \in M_N^{st}(2, d)$ 
such that $i(E, \phi)$ is not $\mu$-stable. 
Then by the proof of \cite[Lemma 3.4]{kk23}, 
the maximal destabilizing subsheaf $F \subset E$ satisfies 
$\Hom(F, E/F \otimes L^{-1} \otimes \omega_C) \neq 0$. 
Under the current assumption of $\rk(E)=2$, 
$F$ and $E/F$ are line bundles with $\deg(F) > \deg(E/F)$. 
Moreover, we have 
$\deg(L^{-1} \otimes \omega_C) \leq 2g(C)-2-(2g(C)-1) <0$. 
This leads to a contradiction. 
\end{proof}

\begin{rmk}
The isomorphism (\ref{eq:dchart}) also preserves $d$-critical structures, see 
\cite[Proposition 3.6]{kk23}. 
\end{rmk}

\subsection{Rigidity of vector bundles}
\begin{defin}
Let $N$ be a rank two vector bundle on $C$ with $\det(N) \cong \omega_C$. 
Then $N$ is \textit{$2$-rigid} if 
$\mH^0(\widetilde{C}, l^*N)=0$ for any stable map 
$l \colon \widetilde{C} \to C$ of degree two. 
\end{defin}

\begin{lem}
Fix a line bundle $L$ on $C$ with degree $d \geq 2g(C)-1$. 
Assume that $L$ is a generic element in $\Pic^d(C)$. 
Then a general $[N] \in \bP \Ext^1(L, L^{-1} \otimes \omega_C)$ 
is $2$-rigid. 
\end{lem}
\begin{proof}
By the arguments in \cite[Section 3]{bp06}, 
 stable bundles in the complement of $\Delta(d) \subset \bP \Ext^1(L, L^{-1} \otimes \omega_C)$ for $d \geq 0$, 
where 
\[
\Delta(d) \coloneqq \left\{
[N] \colon 
\begin{aligned}
&\text{there exists a saturated subbundle } S \subset N \text{ with } \\
&\deg(S)=d, \mH^0(S^{\otimes 2}) \neq 0
\end{aligned}
\right\},
\]
are $2$-rigid.  The proof of \cite[Lemma 3.1]{bp06} shows that 
\[
\dim \Delta(d) < 3g(C)-3 < 3g(C)-2
\leq \dim  \bP \Ext^1(L, L^{-1} \otimes \omega_C)
\]
for $0 \leq d \leq g(C)-1$. 

It remains to show that a generic $[N] \in \bP \Ext^1(L, L^{-1} \otimes \omega_C)$ is slope semistable, which follows from Lemma \ref{lem:Nss} below. 
\end{proof}
An earlier version of this paper omitted the check that generic $[N] \in \bP \Ext^1(L, L^{-1} \otimes \omega_C)$ is slope semistable; we thank the anonymous referee for pointing this out, as well as suggesting the proof of the following lemma.
\begin{lem} \label{lem:Nss}
Let $L$ be a line bundle of degree $d \geq g(C)-1$. Assume that $L$ is a generic element in $\Pic^d(C)$. Then a general $[N] \in \bP \Ext^1(L, L^{-1} \otimes \omega_C)$ is $\mu$-semistable. 
\end{lem}
\begin{proof}
Since $\mu$-semistability is an open condition, it is enough to find a line bundle $L \in \Pic^d(C)$ and a point $[N] \in \bP \Ext^1(L, L^{-1} \otimes \omega_C)$ such that $N$ is $\mu$-semistable. 

We may find a line bundle $L_0 \in \Pic^d(C)$ and line bundles $M_0, M_1$ of degree $g(C)-1$ such that there is a surjection $M_0\oplus M_1 \to L_0$. We may then take a line bundle $L' \in \Pic^0(C)$ such that $M_0 \otimes M_1 \otimes (L')^{\otimes 2} \cong \omega_C$. 
Now, set $N \coloneqq (M_0 \oplus M_1) \otimes L'$ and $L \coloneqq L_0 \otimes L'$. Since $M_0$ and $M_1$ have the same degree, $N$ is $\mu$-semistable. Furthermore, we have $L \in \Pic^d(C)$ and there exists a short exact sequence 
\[
0 \to L^{-1} \otimes \omega_C \to N \to L \to 0 
\]
as required. 
\end{proof}

\subsection{Hitchin fibrations and Gopakumar--Vafa invariants}
In this subsection, we fix a line bundle $L$ on $C$ with $\deg(L) \geq 2g(C)-1$ 
and an extension class $[N] \in \bP \Ext^1(L, L^{-1} \otimes \omega_C)$. 
By the Serre duality 
$\Ext^1(L, L^{-1} \otimes \omega_C) \cong \mH^0(L^{\otimes 2})^\vee$, 
the class $[N]$ defines a linear function $f \colon \mH^0(L^{\otimes 2}) \to \bC$. 
To simplify the notation, we put $M_L(r, 1) \coloneqq M^{\mathrm{st}}_L(r, 1)$. 
We will also consider the moduli space of traceless Higgs bundles: 
\[
M^0_L(r, 1) \coloneqq \{(E, \phi) \in M_L(r, 1) \colon \tr(\phi)=0 \}. 
\]

As before, we have the Hitchin fibration 
\[
h \colon M_L(r, 1) \to \widetilde{B}_L \coloneqq \bigoplus_{i =1,2} \mH^0(L^{\otimes i}). 
\]
We denote the restriction of $h$ to the traceless part $M^0_L(r, 1)$ by
\[
h^0 \colon M_L^0(r, 1) \to B_L \coloneqq \mH^0(L^{\otimes 2}). 
\]
We denote by $\tilde{f} \colon \widetilde{B}_L \to \bC$ the composition 
\[
\tilde{f} \colon \widetilde{B}_L \to \mH^0(L^{\otimes 2}) \xrightarrow{f} \bC, 
\]
where the first map is the projection to the direct summand. 

\begin{defin}
\label{GVdef}
\begin{enumerate}
\item Let $X=\Tot_C(N)$ be a local curve. 
We define the \textit{Gopakumar--Vafa invariants} 
$n_{g, r[C]}(X) \in \bZ$ of $X$ for the curve class $r[C]$ 
by the following identity: 
\begin{equation} \label{eq:GVN}
\sum_{i \in \bZ}\chi({}^p\mcH^i(\phi_{\tilde{f}}(\dR h_* \IC_{M_L(r, 1)})))q^i
=\sum_{g \geq 0} n_{g, r[C]}(X) (q^{-\frac{1}{2}}+q^{\frac{1}{2}})^{2g}. 
\end{equation}
\item We define the integers $n_{g, r[C]}(\Tot_C(L))$ by 
the same equation (\ref{eq:GVN}) but removing the vanishing cycle functor $\phi_{\tilde{f}}$ on the left hand side. 
\end{enumerate}
\end{defin}
\begin{rmk}
\begin{enumerate}
\item If one takes $\deg(L) \gg 0$, 
the above definition agrees with the one in \cite{mt18}, see \cite[Proposition 3.10, Equation (3.13)]{kk23}. 
When $r=2$, we can take an arbitrary $L$ with $\deg(L) \geq 2g(C)-1$ due to Corollary \ref{cor:rank2dchart}. 
\item One can replace the moduli space $M(r, 1)$ with $M(r, d)$ for an arbitrary $d \in \bZ$ to define similar invariants $n_{g, r[C], d}(X)$, called the generalised GV invariants \cite{todGVwc}. By the $\chi$-independence proved in \cite{kk23}, we have $n_{g, r[C], d}(X)=n_{g, r[C], 1}(X)$ for all $d$. 
\end{enumerate}
\end{rmk}

\begin{lem}
We have 
\[
\chi({}^p\mcH^i(\phi_{\tilde{f}}(\dR h_* \IC_{M_L(r, 1)})))
=(-1)^{\deg(L)-g+1}
\chi({}^p\mcH^i(\phi_f(\dR h^0_* \IC_{M^0_L(r, 1)}))). 
\]
\end{lem}
\begin{proof}
We have an isomorphism 
\[
M_L(r, 1) \xrightarrow{\sim} 
\mH^0(L) \times M^0_L(r, 1), \quad 
(E, \phi) \mapsto 
(\tr(\phi), (E, \phi-(\tr(\phi)/2)\id)). 
\]
Under the above isomorphism, 
the Hitchin map is written as 
$h=\id_{\mH^0(L)} \times h^0$. 
Then we have 
\begin{align*}
\phi_{\tilde{f}}(\dR h_* \IC_{M_L(r, 1)})
&=\phi_{\tilde{f}}(\IC_{\mH^0(L)} \boxtimes
\dR h^0_*(\IC_{M^0_L(r, 1)})) \\
&\cong \IC_{\mH^0(L)} \boxtimes
\phi_f(\dR h^0_*(\IC_{M^0_L(r, 1)})), 
\end{align*}
where the second isomorphism is a special case of the Thom--Sebastiani isomorphism \cite{MassTS}. 
Noting that $\dim \mH^0(L)=\deg(L)-g+1$, 
we obtain the desired result. 
\end{proof}

\begin{lem}
\label{lem:KS_conical}
Let $P\in\Perv(B_L)$ be a $\bR_{>0}$-equivariant perverse sheaf, and denote by $P_0=i^*P$ the stalk at the origin, where $i\colon 0\hookrightarrow B_L$ is the inclusion.  Then $\chi(P)=\chi(P_0)$.
\end{lem}
\begin{proof}
    By \cite[Prop.3.7.5]{KSsheaves} the adjunction morphism $\mH(B_L,P)\rightarrow \mH(B_L,P_0)$ is an isomorphism, and the result follows.
\end{proof}

\begin{lem}
\label{lem:van=nearby+stalk}
Let $P \in \Perv(B_L)$ be a $\bR_{>0}$-equivariant perverse sheaf. 
Then we have 
\[
\chi(\phi_f(P))=\chi(P_0)+\chi(P|_{U_\varepsilon}[-1]), 
\]
where $P_0$ denotes the stalk at the origin, and 
$U_\varepsilon \coloneqq f^{-1}(\varepsilon)$ (for $0 < \varepsilon \ll 1$) 
is the nearby hyperplane.  
\end{lem}
\begin{proof}
By Lemma \ref{lem:KS_conical}, and from the long exact sequence in cohomology relating the vanishing cycles functor with the nearby cycles functor $\psi_f$
\[
\mH(B_L,\phi_f(P))\rightarrow \mH(B_L,P\lvert_{f^{-1}(0)})\rightarrow \mH(B_L,\psi_f(P))\rightarrow
\]
it is sufficient to prove that $\chi(\psi_f(P))=\chi(P\lvert_{U_{\epsilon}})$.  By Lemma \ref{lem:KS_conical} again, this is equivalent to showing that 
\[
    \chi((\psi_f(P))_0)=\chi(P\lvert_{U_{\epsilon}})
\]
Let $s\colon f^{-1}(\bR_{>0})\hookrightarrow B_L$ and $z\colon f^{-1}(0)\hookrightarrow B_L$ be the inclusions. Recall also that we denote $i \colon 0 \hookrightarrow B_L$ the inclusion. Then by the definition of the nearby cycles functor, and one more application of Lemma \ref{lem:KS_conical} we have 
\begin{align*}
\mH(B_L,i_*i^*\psi_f P)=&\mH(B_L,i_*i^*z_*z^*s_*s^*P)\\
\cong&\mH(B_L,s_*s^*P)\\
\cong&\mH(f^{-1}(\bR_{>0}),s^*P).
\end{align*}
Finally, since $s^*P$ is $\bR_{>0}$-equivariant, the restriction map
\[
\mH(f^{-1}(\bR_{>0}),s^*P)\rightarrow \mH(U_{\epsilon},P\lvert_{U_{\epsilon}})
\]
is an isomorphism.
\end{proof}
\begin{rmk}
Since we assume that $P$ is $\bC^*$-equivariant in the above lemma, it follows that $P\lvert_{U_{\epsilon}}[-1]$ is perverse.
\end{rmk}
Motivated by the above lemma, we define the following invariants: 
\begin{defin}
Let $U_\varepsilon =f^{-1}(\epsilon)\subset B_L$ be a nearby hyperplane, and let 
$h|_{U_\varepsilon} \colon (h^{0})^{-1}(U_\varepsilon) \to U_\varepsilon$ 
be the restriction of the Hitchin fibration. 
We define the integers $n_{g}(U_{\varepsilon})$ by the following identity:  
\[
\sum_{i \in \bZ}\chi({}^p\mcH^i((\dR (h^0|_{U_\varepsilon})_* \IC_{(h^{0})^{-1}(U_\varepsilon)})))q^i
=\sum_{g \geq 0} n_{g}(U_\varepsilon) (q^{-\frac{1}{2}}+q^{\frac{1}{2}})^{2g}. 
\]
\end{defin}
\begin{prop}
\label{prop:GV=L+U}
For $[N] \in \bP \Ext^1(L, L^{-1} \otimes \omega_C)$, 
we have 
\[
n_{g, r[C]}(\Tot_C(N))=
n_{g, r[C]}(\Tot_C(L))
+n_{g}(U_{\varepsilon}). 
\]
\end{prop}
\begin{proof}
By the decomposition theorem, we have 
\[
\dR h^0_* (\IC_{M^0_L(r, 1)})
\cong\bigoplus_i P_i[i] 
\]
for some perverse sheaves $P_i \in \Perv(B_L)$. 
By Lemma \ref{lem:van=nearby+stalk}, 
we have 
\begin{equation}
\label{eq:decompPi}
\chi(\phi_f(P_i))=\chi((P_i)_0)+\chi((P_i)|_{U_\varepsilon}[-1]) 
\end{equation}
for each $i$. 
By the $\bC^*$-equivariance of the Hitchin fibration, 
we have 
$\chi((P_i)_0)=\chi(P_i)$ for each $i$ by Lemma \ref{lem:KS_conical}. 
Furthermore, $(P_i)|_{U_\varepsilon}[-1]$ is again 
a perverse sheaf on $U_\varepsilon$, 
and we have 
\[
\dR (h^0|_{U_\varepsilon})_* (\IC_{(h^0)^{-1}(U_\varepsilon)})
\cong \bigoplus_i (P_i)|_{U_\varepsilon}[-1]. 
\]

In summary, 
the first term (resp. the second term) 
of the right hand side 
of (\ref{eq:decompPi}) computes 
the invariants 
$n_{g, r[C]}(\Tot_C(L))$ 
(resp. $n_g(U_\varepsilon)$). 
\end{proof}

\begin{rmk}
\label{rmk:strategy}
Assume that $r=2$. 
By Proposition \ref{prop:GV=L+U}, 
the computation of $n_{g, 2[C]}(\Tot_C(N))$ 
is divided into two parts: 
\begin{enumerate}
\item Computation of 
$n_{g, 2[C]}(\Tot_C(L))$: 
this can be done by using the explicit description of the global cohomology 
$\mH^*(M_L(2, 1),\bQ)$ by generators and relations \cite{ht03, ht04}. 
The cohomological and the perverse degree of each generator is determined in \cite{dCHM12}. 

\item Computation of $n_g(U_\varepsilon)$: 
this part can be done by determining the fibre-wise contributions. 
Since $U_\varepsilon$ avoids the origin, 
the spectral curve $C_a$ is \textit{reduced} at every point of $a \in U_\varepsilon$. 
Hence, the fibrewise contributions are computed by certain HOMFLY polynomials \cite{ms13, mau16, my14, os12}. 
We further need to calculate the stratification of $U_\varepsilon$ 
induced by the types of singularities of the spectral curves. 
This last step requires a careful analysis of the geometry of certain degeneracy loci. This analysis becomes increasingly complicated as we allow the genus to increase, and for this reason in this paper we only give the full result in the case $g(C)=2$.
\end{enumerate}
\end{rmk}

\section{Genus two} \label{sec:genustwo}
Let $C$ be a smooth projective curve of genus two. 
Let $L$ be a line bundle on $C$ of degree $3=2g(C)-1$. 
Let $f \colon \mH^0(L^{\otimes 2}) \to \bC$ 
be a linear function. 
The goal of this section is 
to compute the GV invariants associated to $L$ and $f$. 
We apply the strategy explained in Remark \ref{rmk:strategy}. 

\subsection{Stratification of the Hitchin base}
We have the following diagram, in which the square is Cartesian:
\begin{equation*}
\begin{tikzcd}
U_\varepsilon=f^{-1}(\epsilon) \ar[r, hook] \ar[rd, hook]
& B_L^*\coloneqq \mH^0(L^{\otimes 2}) \setminus \{0\} \ar[d] 
& \\
& \bP \mH^0(L^{\otimes 2}) \ar[r] \ar[d, hook] 
& \{L^{\otimes 2} \} \ar[d, hook] \\
& \Sym^6(C) \ar[r, "\AJ^6"] 
&\Pic^6(C). 
\end{tikzcd}
\end{equation*}
We define $H_\infty \coloneqq \bP \mH^0(L^{\otimes 2}) \setminus U_\varepsilon$.  By a partition $\lambda\vdash n$ of a number $n\in\bZ_{\geq 1}$ we mean a weakly descending sequence $(\lambda_1,\ldots,\lambda_{r(\lambda)})$ of integers $\lambda_i\in \bZ_{\geq 1}$ such that $\sum_{i=1}^{r(\lambda)}\lambda_i=n$.  We write $\mu\vdash\lambda$ to mean that the partition $\mu$ is a refinement of the partition $\lambda$, i.e. there is a partition of the set $\{1,\ldots,r(\mu)\}$ so that $\lambda$ is the partition of $n=\lvert \mu\lvert$ induced by $\mu$ along with this partition of sets.  We have a natural stratification $\{S_\lambda\}_{\lambda\vdash 6}$ of $\Sym^6(C)$, indexed by partitions of $6$: 
\begin{equation*}
S_\lambda \coloneqq \{
\lambda_1 x_1 + \cdots +\lambda_m x_m : x_i \in C, x_i \neq x_j \text{ for } i \neq j
\}. 
\end{equation*}
The closures of the strata are given by 
\begin{equation*}
\overline{S}_\lambda = \{
\lambda_1 x_1 + \cdots +\lambda_m x_m : x_i \in C
\}
=\coprod_{\lambda \vdash \mu} S_\mu. 
\end{equation*}

\subsubsection*{Overview}
Here we explain a general strategy for computing the Euler characteristic 
of the stratum $U_\varepsilon \cap S_\lambda$, 
which also works for $g(C)>2$.
\begin{enumerate}
\item We first construct a map 
$\Psi\colon Z \to \bP \mH^0(L^{\otimes 2}) \cap \overline{S}_\lambda$ from a smooth variety $Z$, as follows: 
Consider the diagram 
\[
\begin{tikzcd}
& &C \times C^{\times m} 
\ar[ld, "\sigma"'] \ar[rd, "\mu"]& \\
&C & &C^{\times m}. 
\end{tikzcd}
\]
On $C^{\times m}$, we have a natural morphism of vector bundles: 
\begin{equation}
\label{eq:degenmap}
\alpha \colon 
\mH^0(L^{\otimes2}) \otimes \mcO_{C^{\times m}} \to 
\mu_*(\sigma^*L^{\otimes 2} \otimes \mcO_{\lambda_1 \Delta_1+ \cdots +\lambda_m \Delta_m}), 
\end{equation}
where 
\[
\Delta_i \coloneqq \{
(x, (y_j)_{1\leq j\leq m}) \in C \times C^{\times m} \colon x=y_i\}. 
\]
The restriction of $\alpha$ to a point 
$(y_j)_{1\leq j\leq m} \in C^{\times m}$ is the natural map 
\[
\mH^0(L^{\otimes 2}) \to 
L^{\otimes 2}|_{\lambda_1y_1 + \cdots + \lambda_m y_m}. 
\]
We define $Z$ to be the $(\dim \mH^0(L^{\otimes 2})-1)$-th degeneracy locus of $\alpha$ 
(This variety $Z$ is called 
a \textit{generalized de Jonqui{\`e}res divisor} in the literature \cite{far23, ung21}). 
By construction, 
we have a natural map 
$\Psi\colon Z \to \bP \mH^0(L^{\otimes 2}) \cap \overline{S}_\lambda$. 

\item Next, we remove the pull back $D$ of $H_\infty$ from $Z$. 
$D$ can be similarly described as a degeneracy locus, by replacing 
$\mH^0(L^{\otimes 2})$ with 
its hyperplane $\{f=0\}$ in (\ref{eq:degenmap}). 

\item We then remove the deeper strata 
$S_\mu$, $\lambda \vdash \mu$. 
We have that 
\[
\Psi^{-1}(S_\mu \cap U_\varepsilon) \to 
S_\mu \cap U_\varepsilon
\]
is unramified of degree $d_{\mu,\lambda}$, 
where $d_{\mu,\lambda}$ is the number of ways 
to refine $\mu$ to $\lambda$. 
Finally, we obtain 
\[
d_\lambda \cdot \chi(S_\lambda \cap U_\varepsilon)
=\chi(Z)-\chi(D)-\sum_{\mu \vdash \lambda} d_{\mu,\lambda} \cdot \chi(S_\mu \cap U_\varepsilon) 
\]
\end{enumerate}
where $d_{\lambda}$ is the size of the automorphism group of $\lambda$, i.e.
\[
d_{\lambda}=\prod_{n\geq 1}\lvert \{i\geq 1:\lambda_i=n\}\lvert !.
\]
We have used the fact that if $S'\rightarrow S$ is an unramified $n:1$ cover of a variety $S$, then $\chi(S')=n\chi(S)$.  In what follows we will denote such $S'$ by $S^{[n\sco 1]}$; although there may be many such covers of $S$, the ambiguity in the notation will be irrelevant as long as we are only interested in the Euler characteristic of $S'$.
\begin{ex}
\begin{enumerate}
\item Suppose that $\bP \mH^0(L^{\otimes 2}) \cap S_\lambda$ is $0$-dimensional. 
Since $\bP \mH^0(L^{\otimes 2}) \subset \Sym^{2g(C)-1}(C)$ has codimension $g=\dim J(C)$ and the dimension of $S_\lambda$ 
is equal to the number of rows in $\lambda$, 
this happens precisely when 
\[
\#\{\text{rows in } \lambda\}=g(C). 
\]
In this case, 
\[
\chi(U_\varepsilon \cap S_\lambda)=\chi(Z)/d_\lambda, 
\]
is given by the classical de Jonqui{\`e}res formula \cite{deJ66}. 

\item Suppose that 
\[
k \coloneqq \#\{i \colon \lambda_i=1\} > g(C).  
\]
In this case, slightly modifying the above construction, we define $Z'$ via the Cartesian diagram: 
\[
\begin{tikzcd}
&Z' \ar[r] \ar[d] &\Sym^{k}(C) \ar[d] \\
&C^{\times (m-k)} \ar[r, "\eta"] &\Pic^k(C),
\end{tikzcd}
\]
where $\eta \colon C^{\times (m-k)} \to \Pic^k(C)$ is defined by 
\[
\eta((x_i))
=L^{\otimes 2}(-\lambda_1x_1- \cdots -\lambda_{m-k}x_{m-k}). 
\]
Since $k> g(C)$, the Abel--Jacobi map 
$\Sym^k(C) \to \Pic^k(C)$ is a projective space bundle, hence so is 
$Z' \to C^{\times (m-k)}$. 
It is easier to compute $\chi(Z')$ 
than to compute $\chi(Z)$ for a degeneracy locus $Z$. 

As in the strategy explained above, we then remove $D' \subset Z'$, the pull back of $H_\infty$. It will turn out that $D'$ is also a projective space bundle over $C^{\times (m-k)}$, one dimension less than $Z'$. Finally, we will remove the deeper strata from $Z'$ by the same formula as above. 

\end{enumerate}
\end{ex}
\begin{prop}
\label{prop:strHitchin}
Let $L$ and $f$ be general. 
Then we have 
\begin{equation} 
\label{eq:strHitchin}
\chi(U_\varepsilon \cap S_\lambda)=\begin{cases}
0 & (\lambda=(6)), \\
50 & (\lambda=(5, 1)), \\
128 & (\lambda=(4, 2)), \\
-216 & (\lambda=(4, 1, 1)), \\
81 & (\lambda=(3, 3)), \\
-668 & (\lambda=(3, 2, 1)), \\
542 & (\lambda=(3, 1, 1, 1)), \\
-128 & (\lambda=(2, 2, 2)), \\
968 & (\lambda=(2, 2, 1, 1)), \\
-1012 & (\lambda=(2, 1, 1, 1, 1)). 
\end{cases}
\end{equation}
\end
{prop}

\subsubsection{$\lambda=(6)$}
We want to find $x \in C$ such that $\mcO(6x) \cong L^{\otimes 2}$. 
The locus $\{\mcO(6x) : x \in C\} \subset \Pic^6(C)$ 
is the image of the following composition: 
\begin{equation*}
C \cong \Theta \subset 
\Pic^1(C) \xrightarrow{(-)^{\otimes 6}}
\Pic^6(C), 
\end{equation*}
i.e., a divisor in $\Pic^6(C)$. 
Hence, by choosing $L \in \Pic^3(C)$ general, 
we have $L^{\otimes 2} \notin \{\mcO(6x) : x \in C\} \subset \Pic^6(C)$, 
i.e., 
$S_{(6)} \cap \bP \mH^0(L^{\otimes 2}) = \emptyset$.

\subsubsection{$\lambda=(5, 1)$}
Consider the following composition: 
\begin{equation*}
\begin{tikzcd}
\eta \colon 
\Pic^1(C) \ar[r, "(-)^{\otimes 5}"] 
& \Pic^5(C) \ar[rr, "L^{\otimes 2} \otimes (-)^{-1}"] 
&&\Pic^1(C). 
\end{tikzcd}
\end{equation*}
We have 
\begin{equation*}
S_{(5, 1)} \cap \bP \mH^0(L^{\otimes 2})
=\{(x, y) \in \Theta \times \Theta : \eta(x)=y\}
=\{50 \text{ points} \}, 
\end{equation*}
where $50=\eta_* \Theta . \Theta=5^2\Theta^2$. 
Note also that we have chosen a general $L$ so that 
$\eta(\Theta)$ and $\Theta$ intersect properly. By taking the linear function $f$ to be generic, 
the divisor $H_\infty$ does not intersect with 
the finite set of points $S_{(5, 1)} \cap \bP \mH^0(L^{\otimes 2})$.

\subsubsection{$\lambda=(4, 2)$}
Similarly to the case of $\lambda=(5, 1)$, we get 
$S_{(2, 4)} \cap U_\varepsilon=\{128 \text{ points}\}$, 
where $128=4^2 \cdot 2^2 \cdot \Theta^2$.

\subsubsection{$\lambda=(4, 1, 1)$}
Let $Z \subset C \times \Sym^2(C)$ be the subvariety defined by 
$\mcO_C(4x+y+z)\cong L^{\otimes 2}$. 
We claim that there is an isomorphism 
$\psi \colon C \rightarrow Z$ 
sending $x$ to the inverse under $\Psi$ of the unique section of $L^{\otimes 2}$ 
vanishing to order (at least) 4 at $x$.  
This follows from the following fact: for all $x\in C$ there is an equality $h^0(L^{\otimes 2}(-4x))=1$. 
This in turn follows from genericity of $L$ and the Riemann-Roch theorem. 

We calculate
\[
Z \cap H_\infty=36
\]
as follows. 
The divisor $H_{\infty}$ is linearly equivalent to the divisor $D_p$ defined 
by the vanishing of the section $s$ at $p$, for some fixed $p\in C$.  
If a section $s=\psi(q)$ 
in $Z$ vanishes at $p$, 
either $p=x$, in which case the intersection multiplicity is $4$, 
since the regular function defining $D_p$ vanishes to order 4,
or $p$ is one of the other two zeros of $\Psi(q)$. 
Counting these points corresponds to counting solutions to 
$L^{\otimes 2}(-p)\cong \mathcal{O}_C(4x+y)$. 
But now this is the same theta divisor calculation as we have seen before, 
and we find that there are $4^2*1^2*2$ such points.  
So $Z \cap H_{\infty}=32+4=36$. 

We also need to remove deeper strata from $Z$. 
We find that it contains one copy each of $S_{(5,1)}$ and $S_{(4,2)}$. 
So finally we calculate
\[
\chi(S_{(4,1,1)} \cap U_\varepsilon)=-2-36-50-128=-216.
 \]

\subsubsection{$\lambda=(3, 3)$}
As in the case of $\lambda=(5, 1)$, 
the number of solutions to $\mcO_C(3x+3y) \cong L^{\otimes 2}$ is
\[
3^2 \cdot 3^2 \cdot \Theta^2=162. 
\]
Here we are double counting: 
two solutions $(x,y)$ and $(x',y')$ with $x=y'$ and $y=x'$ 
should only count for one point.  So we find 
\[
\chi(S_{(3, 3)} \cap U_\varepsilon)=81. 
\]

\subsubsection{$\lambda=(3, 2, 1)$}
Let $Z \subset C \times C \times C$ be the subvariety 
defined by $\mcO_C(3x+2y+z) \cong L^{\otimes 2}$. 
We claim that $\chi(Z)=-196$. 
We postpone the proof until the next subsection, see Proposition \ref{prop:Euler321}. 

Fixing some $p\in C$, the preimage $D$ of the divisor $H_{\infty}$ is equivalent to the divisor 
$3\{p\}\times C\times C+2C\times \{p\}\times C+C\times C\times \{p\}$. 
So $D$ consists of $132=3*8+2*18+72$ points.  
These numbers come from the same calculation of intersections of theta divisors as before, so for example $72$ appears via $72=2*2^2*3^2$.  $Z$ also contains some copies of the deeper strata: 
it contains one copy of $S_{(4,2)}$ (corresponding to $z=x$), 
one copy of $S_{(5,1)}$ given by setting $x=y$, 
and a variety $S^{[2\sco 1]}_{(3,3)}$ (corresponding to $y=z$). 
So putting everything together we find
\[
\chi(S_{(3,2,1)} \cap U_\varepsilon)=-196-132-128-50-162
=-668.
\]

\subsubsection{$\lambda=(3, 1^3)$}
Consider $Z$, the subvariety of pairs $(x,(y,z,w))\in C\times \Sym^3(C)$ 
giving solutions to $\mathcal{O}_C(3x+y+z+w)\cong L^{\otimes 2}$. 
Projecting along the morphism that forgets $y,z,w$, 
this is a $\mathbb{P}^1$-bundle $\pi\colon Z\rightarrow C$, and so has Euler characteristic $-4=2*(-2)$.

We claim that the pull back $D$ of $H_\infty$ to $Z$ is isomorphic to $C$. 
To show this, it is enough to prove that $D.F=1$, 
where $F$ is the fibre class. 
The fibre of $\pi$ over a point $x_0 \in C$ is 
$\pi^{-1}(x_0)=\bP \mH^0(L^{\otimes 2}(-3x_0))$. 
The map 
$\pi^{-1}(x_0) \rightarrow \bP \mH^0(L^{\otimes 2})$ 
is induced by a natural embedding 
$u\colon \mH^0(L^{\otimes 2}(-3x_0)) \hookrightarrow \mH^0(L^{\otimes 2})$ and hence is linear, and meets the hyperplane $H_{\epsilon}$ at a single point.  Since $f$ is generic, and $C$ is one dimensional, the image of $u$ intersects $H_{\epsilon}$ transversally for all $x_0$.

Next we remove deeper strata: 
one copy of $S_{(3,2,1)}$, one copy of $S_{(4,1,1)}$, 
a variety $S^{[2\sco 1]}_{(3,3)}$, one copy of $S_{(4,2)}$, and one copy of $S_{(5,1)}$. 
Putting everything together we find 
\begin{align*}
\chi(S_{(3,1,1,1)} \cap U_\varepsilon)
&=-4-(-2)-(-668)-(-216)-2*81-128-50 \\
&=542. 
\end{align*}

\subsubsection{$\lambda=(2^3)$}
Consider the commutative diagram: 
\begin{equation*}
\begin{tikzcd}
\Sym^3(C) \ar[r, hook, "2(-)"] \ar[d, "\AJ^3"] 
&\Sym^6(C) \ar[d, "\AJ^6"] \\
\Pic^3(C) \ar[r, "(-)^{\otimes 2}"] 
&\Pic^6(C). 
\end{tikzcd}
\end{equation*}
The closure $\overline{S}_{(2^3)} \cap \bP \mH^0(L^{\otimes 2})$ is 
the fibre $\eta^{-1}(L^{\otimes 2})$, where 
$\eta=\AJ^6 \circ 2(-)=(-)^{\otimes 2} \circ \AJ^3$.  The morphism $\AJ^3$ is a $\bP^1$-fibration. 
Hence, we have 
\begin{equation*}
\overline{S}_{(2^3)} \cap \bP \mH^0(L^{\otimes 2})=
\coprod_{16} \bP^1. 
\end{equation*}
The embeddings $\bP^1 \hookrightarrow \bP \mH^0(L^{\otimes 2})$ 
are quadric, induced by 
\begin{equation*}
\mH^0(M) \hookrightarrow \mH^0(L^{\otimes 2}), \quad 
s \mapsto s^{\otimes 2}, 
\end{equation*}
where $M$ is a choice of a square root of $L^{\otimes 2}$.
Hence, $\bP^1 \cap H_{\infty}=\{2 \text{ points}\}$. 

We conclude that 
\begin{equation*}
\overline{S}_{(2^3)} \cap U_\varepsilon = 
\coprod_{16} (\bP^1 \setminus \{2 \text{ points}\}). 
\end{equation*}
By removing the deeper stratum $S_{(4, 2)}$, we obtain 
\[
\chi(U_\varepsilon \cap S_{(2^3)})=
16 (\chi(\bP^1)-2)-128=-128.  
\]

\subsubsection{$\lambda=(2, 2, 1, 1)$}
Let $Z \subset C \times C \times \Sym^2(C)$ be the subvariety defined by 
$\mcO_C(2x+2y+z+w) \cong L^{\otimes 2}$. 
The projection $\pi \colon Z \to C\times C$ is given by forgetting $z$ and $w$, 
and is the blowup of $C^{\times 2}$ at 32 points 
corresponding to solutions to $L^{\otimes 2}(-2x-2y)\cong \omega_C$.  
So 
\[
\chi(Z)=4+32=36. 
\]
Let $D$ be the pull back of $H_\infty$ via $Z \to \bP \mH^0(L^{\otimes 2})$. 

\begin{lem}
We have $\chi(D)=-176$. 
\end{lem}
\begin{proof}
Consider the rational map $\phi \colon C \times C \dashrightarrow \bP \mH^0(L^{\otimes 2})$ 
which sends generic $(x, y) \in C \times C$ to the one-dimensional subspace $\mH^0(L^{\otimes 2}(-2x-2y))\subset \mH^0(L^{\otimes 2})$. 
As discussed above, the indeterminacy locus of this rational map is $32$ points in $C \times C$. 
In the following, we calculate the line bundle inducing the rational map $\phi$. 
Consider the diagram: 
\[
\begin{tikzcd}
& &C \times C^{\times 2} \ar[ld, "\sigma"'] \ar[rd, "\mu"] & \\
&C & &C^{\times 2}, 
\end{tikzcd}
\]
where $\sigma$ is the first projection and 
$\mu$ is the projection forgetting the first factor. 
Then the line bundle inducing the rational map $\phi$ is the double dual of 
$\mcF\coloneqq \mu_*(\sigma^*L^{\otimes 2}(-2\Delta_1-2\Delta_2))$, where 
\[
\Delta_i \coloneqq \{(x, (y_1, y_2)) \in C \times C^{\times 2} \colon x=y_i\}, \quad i=1, 2. 
\]
We calculate the first Chern character of $\mcF$ by the Grothendieck-Riemann-Roch formula. 
Let $\eta \in \mH^2(C \times C^{\times 2}, \bQ)$ denote the dual class of $\pt \times C^{\times 2}$. From now on we freely refer to cohomology classes by using their dual homology classes.
We have 
\begin{align*}
&\quad \ch(\sigma^*L^{\otimes 2}(-2\Delta_1-2\Delta_2)).\td_\mu \\
&=\left(1, -2\Delta_1-2\Delta_2, \frac{4(\Delta_1+\Delta_2)^2}{2}, * \right).(1, 6\eta, 0, 0).(1, -\eta, 0, 0) \\
&=(1, 5\eta-2\Delta_1-2\Delta_2, *_2, *_3), 
\end{align*}
where 
\begin{align*}
*_2
&=2(\Delta_1+\Delta_2)^2-10\eta.(\Delta_1+\Delta_2) \\
&=2(-2C_3+2\delta-2C_2)-10(C_3+C_2) \\
&=-14C_3-14C_2+4\delta, 
\end{align*}
and $\delta$ denotes the class of the small diagonal, 
$C_2$ denotes the class of $\pt \times C \times \pt$, and 
$C_3$ is the class of $\pt \times \pt \times C$. 
The intersection $\Delta_1^2=-2C_3$ can be calculated as follows: 
\[
\Delta_1^2
=p_{12}^*(\Delta^2)
=-2p_{12}^*(\pt), 
\]
where $p_{12} \colon C \times (C^{\times 2}) \to C \times C$ is given by forgetting the third factor, 
and $\Delta \subset C \times C$ is the diagonal. The other parts of the calculations for $*_2$ are similar. 

Hence, we obtain 
$\ch_1(\mcF)=-14(f_1+f_2)+4\Delta$, and 
\[
\phi^*\mcO(1)=\mcO_{C \times C}(14(f_1+f_2)-4\Delta)
\]
where $f_1$ is the class of $C\times \pt$ and $f_2$ is the class of $\pt\times C$.
Now, the map $Z \to \bP \mH^0(L^{\otimes 2})$ is induced by the strict transform of 
$\pi^*(14(f_1+f_2)-4\Delta)$ and hence 
\[
D=\pi^*(14(f_1+f_2)-4\Delta)-e, 
\]
where $e$ denotes the exceptional divisor of $\pi$. 
We obtain: 
\begin{align*}
-\chi(D)
=&\deg K_{D} \\
=&(K_Z+D).D \\
=&((\pi^*(2f_1+2f_2)+e)+(\pi^*(14(f_1+f_2)-4\Delta)-e)) \\ &.(\pi^*(14(f_1+f_2)-4\Delta)-e) \\
=&(16(f_1+f_2)-4\Delta).(14(f_1+f_2)-4\Delta) \\
=&2 \cdot 16 \cdot 14-2 \cdot (16+14) \cdot 4 +4^2 \cdot (-2)\\
=&176. 
\end{align*}
\end{proof}
$Z$ contains points corresponding to deeper strata, 
for example it contains three copies of $S_{(4,2)}$ 
given by either setting $z=w=x$, or $z=w=y$, or $x=y$ and $z=w$. 

In all, we must remove three copies of $S_{(4,2)}$, 
a double cover $S^{[2\sco 1]}_{(3,3)}$, two copies of $S_{(3,2,1)}$, 
a cover $S^{[6\sco 1]}_{(2,2,2)}$, one copy of $S_{(5,1)}$, 
and one copy of $S_{(4,1,1)}$.  
The number six deserves some explanation. 
Let $(a,b,c)$ be a point of $\Sym^6(C)$, lying in the image of $S_{(2,2,2)}$. 
Then we can lift it to a point of $Z$ in six different ways, the first by setting $x=a,y=b,z=w=c$, 
and the other five by permuting $a,b,c$ in these three equalities.  
Finally, note that after removing $T$, 
the union of all of these contributions from deeper strata, 
the cover $Z\setminus T\rightarrow S_{(2,2,1,1)}$ is $2:1$.  
Putting this all together, we find
\begin{align*}
&\quad \chi(S_{(2,2,1,1)}) \\
&=\left(36-(-176)-3*128-2*81-2*(-668)-6*(-128)-50-(-216)\right)/2 \\
&=968.
\end{align*}

\subsubsection{$\lambda=(2, 1^4)$}
We consider solutions $Z\subset C\times \Sym^4(C)$ 
to $\mathcal{O}_C(2x+y+z+w+t)\cong L^{\otimes 2}$.  
This is a $\mathbb{P}^2$-bundle $r\colon Z\rightarrow C$, where $r$ forgets $y,z,w,t$. 
So 
\[
\chi(P)=3*(-2)=-6.  
\]
Let $D$ be the pull back of $H_\infty$ to $Z$, 
and $F$ a fibre of $r$. 
As in the case $\lambda=(3, 1, 1, 1)$, 
we can see that $D \cap F$ is isomorphic to $\bP^1$, linearly embedded into $F$. 
In other words, $D$ is a $\bP^1$-bundle over $C$ 
and so $\chi(D)=-4=(-2)*2$. 

We need to remove the deeper strata: 
\[
S^{[2\sco 1]}_{(2,2,1,1)}+ S_{(3,1,1,1)}+ S^{[3\sco 1]}_{(2,2,2)}+2S_{(3,2,1)}+S_{(4,1,1)}+S^{[2\sco 1]}_{(3,3)}+2S_{(4,2)}+S_{(5,1)}
\]
The degree of a cover $mS^{[n\sco 1]}_{\pi}\rightarrow S_{\pi}$ is $mn$.  The above degrees are given as follows: a point in, say $S_{(3,2,1)}$, is represented by a tuple $(x_0,\ldots ,x_5)\in\Sym^6(C)$ with $x_0=x_1=x_2$ and $x_3=x_4$.  In order to lift to $Z$ we need to pick one out of $x_0$ or $x_3$ to be the value of $x$.  In general the degree of the cover over $S_{\pi}$ is given by counting the number of rows in $\pi$ of length at least $2$. 

Putting this all together we calculate
\begin{align*}
\chi(S_{(2,1^4)} \cap U_\varepsilon)
=&-6-(-4)-2*968-542-3*(-128)-2*(-668)\\
&-(-216)-2*81-2*128-50\\
=&-1012. 
\end{align*}

\subsection{A degeneracy locus calculation}
Let $Z \subset C \times C \times C$ be the subvariety 
defined by $\mcO_C(3x+2y+z) \cong L^{\otimes 2}$. 
The goal of this subsection is to compute the Euler characteristic of $Z$. 
To do so, we first describe $Z$ as a degeneracy locus. 
Consider the following diagram: 
\[
\begin{tikzcd}
& &C \times C^{\times 3} \ar[ld, "\sigma"'] \ar[rd, "\mu"] & \\
&C & &C^{\times 3}, 
\end{tikzcd}
\]
where $\sigma$ denotes the first projection and 
$\mu$ denotes the projection forgetting the first factor. 
On $C \times C^{\times 3}$, we have a natural map of vector bundles: 
\[
\sigma^*L^{\otimes 2} \to 
\sigma^*L^{\otimes 2} \otimes \mcO_{3\Delta_1+2\Delta_2+\Delta_3}, 
\]
where 
\[
\Delta_i \coloneqq \{(x, (y_1, y_2, y_3)) \in C \times C^{\times 3} \colon x=y_i\} 
\]
for $i=1, 2, 3$. 
Pushing this down via $\mu$, we obtain 
\[
\alpha \colon 
\mH^0(C, L^{\otimes 2}) \otimes \mcO_{C^{\times 3}} \to 
\mcF\coloneqq \mu_*(\sigma^*L^{\otimes 2} \otimes \mcO_{3\Delta_1+2\Delta_2+\Delta_3}). 
\]
The fibre of the vector bundle $\mcF$ at a point $(x, y, z) \in C^{\times 3}$ is 
the $6$ dimensional vector space $L^{\otimes 2}|_{3x+2y+z}$. 

Hence, our variety $Z \subset C^{\times 3}$ is the $4$th degeneracy locus of 
$\alpha$, or equivalently, of $\alpha^\vee$. 

\begin{lem} \label{lem:Eulerdegen}
We have $\chi(Z)=c_2(\mcF)c_1(C^{\times 3})-c_1(\mcF)c_2(\mcF)-c_3(\mcF)$. 
\end{lem}
\begin{proof}
For simplicity, we put $V=\mH^0(C, L^{\otimes 2})$.
We will apply \cite[Example 5.8 (i)]{pra88} 
to the map 
\[
\alpha^\vee \colon \mcF^\vee \to V \otimes \mcO. 
\]
Note that we have 
$m=6, n=5 ,r=4$ in the notations of \textit{loc. cit.}. 
We first fix some more notation. 
For a vector bundle $E$, we denote by 
$s_i(E)=(-1)^ic_i(-E)$ the $i$-th Segre class of $E$, and for two vector bundles $E$ and $E'$, we put 
\[
s_i(E-E')\coloneqq
\sum_{p=0}^i(-1)^{i-p}s_p(E)c_{i-p}(E'). 
\]
Finally, we put 
\[
s_{(2, 1)}(E-E')\coloneqq 
s_{2}(E-E')s_{1}(E-E')-s_{3}(E-E')s_{0}(E-E'). 
\]
Now, by \cite[Example 5.8 (i)]{pra88}, we have 
\begin{align*}
\chi(Z)&=
s_{2}(V \otimes \mcO-\mcF^\vee)c_1(C^{\times 3})
-s_{(2, 1)}(V \otimes \mcO-\mcF^\vee)
-2s_{3}(V \otimes \mcO-\mcF^\vee) \\
&=c_2(\mcF^\vee)c_1(C^{\times 3})
-(-c_2(\mcF^\vee)c_1(\mcF^\vee)+c_3(\mcF^\vee))
+2c_3(\mcF^\vee) \\
&=c_2(\mcF) c_1(C^{\times 3})-c_1(\mcF)c_2(\mcF)-c_3(\mcF)
\end{align*}
as desired. 
\end{proof}

We will use the following cohomology classes to express the Chern classes of $\mcF$: 
\begin{itemize}
\item $\eta\coloneqq \sigma^*\pt$, 
\item $\Delta_i \coloneqq \{(x, (y_1, y_2, y_3)) \in C \times C^{\times 3} \colon x=y_i\}$ for $i=1, 2, 3$, 
\item $\Delta_{ij} \coloneqq \{(x, (y_1, y_3, y_3)) \in C \times C^{\times 3} \colon x=y_i=y_j\}$ 
for $1 \leq i< j \leq 3$, 
\item $C_{ij} \coloneqq \tilde{p}_{ij}^*\pt$, 
where $\tilde{p}_{ij} \colon C \times C^{\times 3} \to C \times C$ denotes 
the projection \textit{forgetting} the $(i+1)$th and $(j+1)$th factors for $1\leq i < j \leq 3$, 
\item $C_i \coloneqq \tilde{p}_i^*\pt$, 
where $\tilde{p}_{i} \colon C \times C^{\times 3} \to C^{\times 3}$ denotes 
the projection forgetting the $(i+1)$th factor for $i=1, 2, 3$, 
\item $\delta$ is the small diagonal in $C \times C^{\times 3}$. 
\end{itemize}
By abuse of notation, we denote the pushforward via $\mu$ of the above classes by the same symbols. We first compute the Chern character of $\mcF$ by using the Grothendieck--Riemann--Roch formula: 
\begin{lem}
We have 
\begin{enumerate}
\item $\ch_0(\mcF)=6$, 
\item $\ch_1(\mcF)=24C_{23}+14C_{13}+6C_{12}-6\Delta_{12}-3\Delta_{13}-2\Delta_{23}$, 
\item $\ch_2(\mcF)=-60C_3-27C_2-16C_1+6\delta$, 
\item $\ch_3(\mcF)=66 \pt$. 
\end{enumerate}
\end{lem}
\begin{proof}
We denote by $*_i$ the $i$th component of 
$\ch(\sigma^*L^{\otimes 2} \otimes \mcO_{3\Delta_1+2\Delta_2+\Delta_3}).\td_\mu$. 
It is obvious that 
\[
*_0=0, \quad *_1=3\Delta_1+2\Delta_2+\Delta_3. 
\]
Noting that 
$\ch(\sigma^*L^{\otimes 2}).\td_\mu=(1, 5\eta, 0, 0 ,0)$, we obtain 
\begin{align*}
*_2
&=-\frac{(3\Delta_1+2\Delta_2+\Delta_3)^2}{2}+5\eta.(3\Delta_1+2\Delta_2+\Delta_3) \\
&=-\frac{1}{2}(-18C_{23}-8C_{13}-2C_{12}+12\Delta_{12}+6\Delta_{13}+4\Delta_{23}) \\
&\quad+15C_{23}+10C_{13}+5C_{13} \\
&=24C_{23}+14C_{13}+6C_{12}-6\Delta_{12}-3\Delta_{13}-2\Delta_{23}. 
\end{align*}
\begin{align*}
*_3
&=\frac{1}{6}(3\Delta_1+2\Delta_2+\Delta_3)^3
-\frac{5}{2}\eta.(3\Delta_1+2\Delta_2+\Delta_3)^2 \\
&=\frac{1}{6}(-18C_{23}-8C_{13}-2C_{12}+12\Delta_{12}+6\Delta_{13}+4\Delta_{23}).
(3\Delta_1+2\Delta_2+\Delta_3) \\
&\quad-\frac{5}{2}\eta.
(-18C_{23}-8C_{13}-2C_{12}+12\Delta_{12}+6\Delta_{13}+4\Delta_{23}) \\
&=\frac{1}{6}(
-24C_3-6C_2+36(-2)C_3+18(-2)C_2+12\delta \\
&~\quad\quad -36C_3-4C_1+24(-2)C_3+12\delta+8(-2)C_1 \\
&~\quad\quad-18C_2-8C_1+12\delta+6(-2)C_2+4(-2)C_1) \\
&\quad+5(-6C_3-3C_2-2C_1) \\
&=-30C_3-12C_2-6C_1+6\delta-30C_3-15C_2-10C_1 \\
&=-60C_3-27C_2-16C_1+6\delta.
\end{align*}
\begin{align*}
*_4
&=-\frac{1}{4} \cdot \frac{1}{6}(3\Delta_1+2\Delta_2+\Delta_3)^4
+\frac{5}{6}\eta.(3\Delta_1+2\Delta_2+\Delta_3)^3 \\
&=-\frac{1}{4}(-30C_3-12C_2-6C_1+6\delta).(3\Delta_1+2\Delta_2+\Delta_3) \\
&\quad+5\eta.(-30C_3-12C_2-6C_1+6\delta) \\
&=-\frac{1}{4}(-18-36-24-24-30-12)+5 \cdot 6 \\
&=66. 
\end{align*}

Most of the above calculations are straightforward. 
Slightly non-trivial ones are, for example: 
\[
\Delta_1^2=-2C_{23}, \quad \Delta_1^3=0, \quad \delta.\Delta_1=-2\pt. 
\]
The first two equalities follow by observing that $\Delta_1=p_{01}^*\Delta$, 
where $p_{01}$ is the projection to the first two factors, 
and $\Delta \subset C \times C$ denotes the diagonal. 
Then the third one follows, as $\delta=\Delta_1.\Delta_{23}$. 

By pushing forward $*_i$ via $\mu$, we obtain the result. 
\end{proof}

\begin{prop} \label{prop:Euler321}
We have 
$\chi(Z)=-196$. 
\end{prop}
\begin{proof}
We need to calculate the Chern classes of $\mcF$ from its Chern character: 
\[
c_2(\mcF)=\frac{1}{2}\ch_1(\mcF)^2-\ch_2(\mcF); \quad 
c_3(\mcF)=\frac{1}{6}\ch_1(\mcF)^3-\ch_1(\mcF)\ch_2(\mcF)+2\ch_3(\mcF). 
\]

Firstly, 
\begin{align*}
\ch_1(\mcF)^2
&=(24C_{23}+14C_{13}+6C_{12}-6\Delta_{12}-3\Delta_{13}-2\Delta_{23})^2 \\
&=24(14C_3+6C_2-6C_3-3C_2-2C_{23}.\Delta_{23}) \\
&\quad+14(24C_3+6C_1-6C_3-3C_{13}.\Delta_{13}-2C_1) \\
&\quad+6(24C_2+14C_1-6C_{12}.\Delta_{12}-3C_2-2C_1) \\
&\quad-6(24C_3+14C_3+6C_{12}.\Delta_{12}-6(-2)C_3-3\delta-2\delta) \\
&\quad-3(24C_2+14C_{13}.\Delta_{13}+6C_2-6\delta-3(-2)C_2-2\delta) \\
&\quad-2(24C_{23}.\Delta_{23}+14C_1+6C_1-6\delta-3\delta-2(-2)C_1) \\
&=144C_3+90C_2+80C_1-96C_{23}.\Delta_{23}-84C_{13}.\Delta_{13}
-72C_{12}.\Delta_{12}+72\delta.
\end{align*}
We then have 
\begin{align*}
\ch_1(\mcF)^3
&=(144C_3+90C_2+80C_1-96C_{23}.\Delta_{23}-84C_{13}.\Delta_{13}
-72C_{12}.\Delta_{12}+72\delta). \\
&\quad (24C_{23}+14C_{13}+6C_{12}-6\Delta_{12}-3\Delta_{13}-2\Delta_{23}) \\
&=24(80-84-72+72)+14(90-96-72+72)+6(144-96-84+72) \\
&\quad-6(90+80-96-84-72(-2)+72(-2)) \\
&\quad-3(144+80-96-84(-2)-72+72(-2)) \\
&\quad-2(144+90-96(-2)-84-72+72(-2)) \\
&=-396, 
\end{align*}
and 
\begin{align*}
\ch_1(\mcF)\ch_2(\mcF)
&=-60(6-3-2)-27(14-6-2)-16(24-6-3) \\
&\quad+6(24+14+6-6(-2)-3(-2)-2(-2)) \\
&=-66. 
\end{align*}
We also have 
\begin{align*}
c_2(\mcF)c_1(C^{\times 3})
&=(\frac{1}{2}\ch_1(\mcF)^2-\ch_2(\mcF))(-2)(C_{23}+C_{13}+C_{12}) \\
&=(132C_3+72C_2+56C_1-48C_{23}.\Delta_{23}-42C_{13}.\Delta_{13}
-36C_{12}.\Delta_{12}+30\delta). \\
&\quad (-2)(C_{23}+C_{13}+C_{12}) \\
&=-2\{(56-42-36+30)+(72-48-36+30)+(132-48-42+30)\} \\
&=-196. 
\end{align*}

Finally, by Lemma \ref{lem:Eulerdegen}, we get 
\begin{align*}
\chi(Z)
&=c_2(\mcF)c_1(C^{\times 3})-c_1(\mcF)c_2(\mcF)-c_3(\mcF) \\
&=c_2(\mcF)c_1(C^{\times 3})
-\frac{2}{3}\ch_1(\mcF)^3+2\ch_1(\mcF)\ch_2(\mcF)-2\ch_3(\mcF) \\
&=-196+264-132-132=-196. 
\end{align*}
\end{proof}

\subsection{Fibrewise GV invariants}
In this subsection, we determine the fibrewise GV invariants associated to the spectral curves over points in $U_\varepsilon$. Given a partition $\lambda\vdash 6$ and an integer $g\geq 0$ we define $n_g(\lambda)$ to be the local contribution of a spectral curve corresponding to a point in $S_{\lambda}\cap U_{\epsilon}$ to the genus $g$ Gopakumar--Vafa invariant; by the formulas recalled in the proof of Proposition \ref{prop:GVfiber}, these contributions only depend on $\lambda$, so that these numbers are well-defined.
In order to do so, we calculate certain HOMFLY polynomials: 
We define the Laurent polynomials 
$P(T_{2, n})$ for $n \geq 0$ as follows: 
\begin{equation} \label{eq:HOMFLYdef}
\begin{aligned}
&P(T_{2, 0})=\frac{a-a^{-1}}{q-q^{-1}}, \\
&P(T_{2, 1})=1, \\
&P(T_{2, n})=-a(q-q^{-1})P(T_{2, n-1})+a^2P(T_{2, n-2}) \text{ for } n \geq 2. 
\end{aligned}
\end{equation}
This last equation is the standard recursion relation for HOMFLY polynomials of torus knots, and so $P(T_{2,n})$ is the HOMFLY polynomial of the torus knot $T_{2,n}$ (see \cite[Section 2]{os12}).  We will use the specialization $(q/a)^{n-1}P(T_{2, n})|_{a=0}$ for $2 \leq n \leq 5$:
\begin{align*}
(q/a)P(T_{2, 2})|_{a=0}=&
	q \cdot \{(q^{-1}-q)+\frac{1}{q^{-1}-q}\}, \\
(q/a)^2P(T_{2, 3})|_{a=0}= &
	q^2 \cdot\{(q^{-1}-q)^2+2\}, \\
(q/a)^3P(T_{2, 4})|_{a=0}=&
	q^3 \cdot \{(q^{-1}-q)^3+3(q^{-1}-q)+\frac{1}{q^{-1}-q}\}, \\
(q/a)^4P(T_{2, 5})|_{a=0}= &
	q^4 \cdot \{(q^{-1}-q)^4+4(q^{-1}-q)^2+3\}. 
\end{align*}

\begin{prop}
\label{prop:GVfiber}
The local contributions to the GV invariants satisfy $n_g(\lambda)=0$ for $g>6$ or $g<3$, and the remaining values are as in the following table
\begin{center}
\begin{tabular}{| c | c c c c| c|}
 \hline
 $\lambda$&$g=3$ & $4$ & $5$&$6$ &$\chi(U_{\epsilon}\cap S_{\lambda})$\\ 
 \hline
 $(5,1)$ & $0$ & $3$&$-4$&$1$&$50$ \\  
 $(4,2)$ & $-1$ &$4$&$-4$&$1$&$128$\\
 $(4,1,1)$&$0$&$1$&$-3$&$1$&$-216$\\
 $(3,3)$&$0$&$4$&$-4$&$1$&$81$\\
 $(3,2,1)$&$0$&$2$&$-3$&$1$&$-668$\\
 $(3,1^3)$&$0$&$0$&$-2$&$1$&$542$\\
 $(2^3)$&$-1$&$3$&$-3$&$1$&$-128$\\
 $(2,2,1,1)$&$0$&$1$&$-2$&$1$&$968$\\
 $(2,1^4)$&$0$&$0$&$-1$&$1$&$-1012$\\
 $(1^6)$&$0$&$0$&$0$&$1$&$256$\\
 \hline
\end{tabular}
\end{center}
We have also included the Euler characteristics of the corresponding strata of $U_{\epsilon}$, given by Proposition \ref{prop:strHitchin}, for reference.

\end{prop}
\begin{proof}
Let $C_a \to C$ be the spectral cover corresponding to a point 
$a \in S_{\lambda} \cap \bP \mH^0(L^{\otimes 2})$. 
For any partition $\lambda$, the spectral curve $C_a$ is reduced. 
Hence, by \cite{ms13, my14}, the fibrewise contribution $n_g(\lambda)$ is computed by 
the Euler characteristics of the Hilbert schemes of points on $C_a$. 
The latter is then computed by using certain HOMFLY polynomials. 
More precisely, \cite{ms13, my14} gives the first, and \cite{os12,mau16}, gives the second equality in the following sequence of equalities:
\begin{align}
\label{Hilb_Euler_formula}
\sum_{g \geq 0}q^{2g(C_a)-2}\!\cdot\!(-1)^{g}n_g(\lambda)(q^{-1}-q)^{2g-2} 
&=\int_{C_a^{[*]}}q^{2l}d\chi \\ \nonumber
&=(1-q^2)^{-\chi(C_a)}\prod_{i=1}^{r(\lambda)}[(q/a)^{\lambda_i-1}P(T_{2, \lambda_i})]_{a=0} \\ \nonumber
    &=q^{-\chi(C_a)}(q^{-1}-q)^{-\chi(C_a)}\prod_{i=1}^{r(\lambda)}[(q/a)^{\lambda_i-1}P(T_{2, \lambda_i})]_{a=0}.
\end{align}
where we denote by $r(\lambda)$ the number of rows of $\lambda$.  We give a little more detail, firstly in order to explain the sign $(-1)^g$ appearing in the LHS of \eqref{Hilb_Euler_formula}. 
Note that $2g(C_a)-2=10$ for all $a$.  By \cite{ms13, my14}, we have the following (cf. \cite{mt18}[Equation (4.4)]): 
\begin{equation*}
\sum_{l}\chi(\IC_{C_a^{[l]}})q^l=
q^{g(C_a)-1} \cdot
\frac{1}{(q^{-1/2}+q^{1/2})^2}\sum_i\chi(\Gr_P^i\IC(\overline{J}))q^i. 
\end{equation*}
After the variable change $q \mapsto q^2$, the LHS is equal to 
\[
\int_{C_a^{[*]}}(-1)^lq^{2l} d\chi, 
\]
and the RHS is equal to 
\[
q^{2g(C_a)-2}\sum_g n_g(\lambda)(q^{-1}+q)^{2g-2}. 
\]
By a further variable change $q^2 \mapsto -q^2$, we obtain the first equation in (\ref{Hilb_Euler_formula}). 

Let us also explain the second equality. Let $\{p_i\}$ be the set of singular points of $C_a$. By \cite[Conjecture 1]{os12} (which is proved in \cite{mau16}), we have 
\[
\int_{C_a^{[*]}}q^{2l}d\chi
=(1-q^2)^{-\chi(C_a)}\prod_{i}[(q/a)^{\mu_i}P(L_i)]_{a=0}, 
\]
where $\mu_i$ denotes the Milnor number of the singularity $p_i$ and $L_i$ denotes the link of $C_a$ at $p_i$. In our case, the singular points of $C_a$ are parametrised by the rows of the partition $\lambda=(\lambda_i)$, and the type of the singular point corresponding to $\lambda_i$ is $\{y^2=x^{\lambda_i}\}$. As a result, we have $\mu_i=\lambda_i-1$ and $L_i=T_{2, \lambda_i}$; these equalities imply the second equality in (\ref{Hilb_Euler_formula}). 

\begin{enumerate}
\item $\lambda=(5, 1)$: 
In this case, 
the spectral curve $C_a$ is a $2:1$ cover branched at two points. 
Hence $\chi(C_a)=2\chi(C)-2=-6$. 
It has a unique singular point of type $\{y^2=x^5\}$. 
Hence we have 
\begin{align*}
\int_{C_a^{[*]}}q^{2l}d\chi
&=q^6(q^{-1}-q)^6 [(q/a)^4P(T_{2,5})]_{a=0} \\
&=q^6(q^{-1}-q)^6 \cdot q^4 \cdot \{(q^{-1}-q)^4+4(q^{-1}-q)^2+3\} \\
&=q^{10}\{(q^{-1}-q)^{10}+4(q^{-1}-q)^8+3(q^{-1}-q)^6\}. 
\end{align*}

\item $\lambda=(4, 2)$: 
We have $\chi(C_a)=-6$. 
$C_a$ has two singular points, of types $\{y^2=x^2\}$ and $\{y^2=x^4\}$. 
We obtain: 
\begin{align*}
\int_{C_a^{[*]}}q^{2l}d\chi
&=q^6(q^{-1}-q)^6 [(q/a)^3P(T_{2,4})]_{a=0} \cdot [(q/a)P(T_{2,2})]_{a=0} \\
&=q^6(q^{-1}-q)^6 
\cdot q^3 \cdot \{(q^{-1}-q)^3+3(q^{-1}-q)+\frac{1}{q^{-1}-q}\} \\
&\quad \cdot q \cdot \{(q^{-1}-q)+\frac{1}{q^{-1}-q}\} \\
&=q^{10}\{(q^{-1}-q)^{10}+4(q^{-1}-q)^8+4(q^{-1}-q)^6+(q^{-1}-q)^4\}. 
\end{align*}

\item $\lambda=(4, 1, 1)$: 
We have $\chi(C_a)=2\chi(C)-3=-7$, 
$C_a$ has a unique singular point of type $\{y^2=x^4\}$. 
We obtain: 
\begin{align*}
\int_{C_a^{[*]}}q^{2l}d\chi
&=q^7(q^{-1}-q)^7 [(q/a)^3P(T_{2,4})]_{a=0} \\
&=q^7(q^{-1}-q)^7 
\cdot q^3 \cdot \{(q^{-1}-q)^3+3(q^{-1}-q)+\frac{1}{q^{-1}-q}\} \\
&=q^{10}\{(q^{-1}-q)^{10}+3(q^{-1}-q)^8+(q^{-1}-q)^6\}. 
\end{align*}

\item $\lambda=(3, 3)$: 
We have $\chi(C_a)=-6$, $C_a$ has two singular points of type $\{y^2=x^3\}$. 
So we obtain: 
\begin{align*}
\int_{C_a^{[*]}}q^{2l}d\chi
&=q^6(q^{-1}-q)^6 ([(q/a)^2P(T_{2,3})]_{a=0})^2 \\
&=q^6(q^{-1}-q)^6 \cdot q^4 \cdot\{(q^{-1}-q)^2+2\}^2 \\
&=q^{10}\{(q^{-1}-q)^{10}+4(q^{-1}-q)^8+4(q^{-1}-q)^6\}. 
\end{align*}

\item $\lambda=(3, 2, 1)$: 
We have $\chi(C_a)=-7$, $C_a$ has two singular points, of types $\{y^2=x^2\}$ and $\{y^2=x^3\}$. 
Hence we obtain: 
\begin{align*}
\int_{C_a^{[*]}}q^{2l}d\chi
&=q^7(q^{-1}-q)^7 [(q/a)^2P(T_{2,3})]_{a=0} \cdot [(q/a)P(T_{2,2})]_{a=0} \\
&=q^7(q^{-1}-q)^7 
\cdot q^2 \cdot\{(q^{-1}-q)^2+2\} 
\cdot q \cdot \{(q^{-1}-q)+\frac{1}{q^{-1}-q}\} \\
&=q^{10}\{(q^{-1}-q)^{10}+3(q^{-1}-q)^8+2(q^{-1}-q)^6\}. 
\end{align*}

\item $\lambda=(3, 1, 1, 1)$: 
We have $\chi(C_a)=-8$, $C_a$ has a unique singular point of type $\{y^2=x^3\}$. 
We obtain: 
\begin{align*}
\int_{C_a^{[*]}}q^{2l}d\chi
&=q^8(q^{-1}-q)^8 [(q/a)^2P(T_{2,3})]_{a=0} \\
&=q^8(q^{-1}-q)^8 
\cdot q^2 \cdot\{(q^{-1}-q)^2+2\} \\
&=q^{10}\{(q^{-1}-q)^{10}+2(q^{-1}-q)^8\}. 
\end{align*}

\item $\lambda=(2, 2, 2)$: 
We have $\chi(C_a)=-7$, $C_a$ has three singular points of type $\{y^2=x^2\}$. 
We obtain: 
\begin{align*}
\int_{C_a^{[*]}}q^{2l}d\chi
&=q^7(q^{-1}-q)^7 ([(q/a)P(T_{2,2})]_{a=0})^3 \\
&=q^7(q^{-1}-q)^7 \cdot q^3 \cdot \{(q^{-1}-q)+\frac{1}{q^{-1}-q}\}^3  \\
&=q^{10}\{(q^{-1}-q)^{10}+3(q^{-1}-q)^8+3(q^{-1}-q)^6+(q^{-1}-q)^4\}. 
\end{align*}

\item $\lambda=(2, 2, 1, 1)$: 
We have $\chi(C_a)=-8$, $C_a$ has two singular points of type $\{y^2=x^2\}$. 
We obtain: 
\begin{align*}
\int_{C_a^{[*]}}q^{2l}d\chi
&=q^8(q^{-1}-q)^8 ([(q/a)P(T_{2,2})]_{a=0})^2 \\
&=q^8(q^{-1}-q)^8 \cdot q^2 \cdot \{(q^{-1}-q)+\frac{1}{q^{-1}-q}\}^2  \\
&=q^{10}\{(q^{-1}-q)^{10}+2(q^{-1}-q)^8+(q^{-1}-q)^6\}. 
\end{align*}

\item $\lambda=(2, 1, 1, 1, 1)$: 
We have $\chi(C_a)=-9$, $C_a$ has a unique singular point of type $\{y^2=x^2\}$. 
We obtain: 
\begin{align*}
\int_{C_a^{[*]}}q^{2l}d\chi
&=q^9(q^{-1}-q)^9 [(q/a)P(T_{2,2})]_{a=0} \\
&=q^9(q^{-1}-q)^9 \cdot q \cdot \{(q^{-1}-q)+\frac{1}{q^{-1}-q}\}  \\
&=q^{10}\{(q^{-1}-q)^{10}+(q^{-1}-q)^8\}. 
\end{align*}

\end{enumerate}
\end{proof}

\subsection{Contributions from the nearby hyperplane}
Putting the discussions in the previous subsections together, 
we obtain: 
\begin{prop} \label{prop:GVnearby}
We have 
\[
n_{g}(U_\varepsilon)=\begin{cases}
1 & (g=6), \\
-8 & (g=5), \\
18 & (g=4), \\
0 & (\text{otherwise}). 
\end{cases}
\]
\end{prop}
\begin{proof}
We have 
\[
n_g(U_\varepsilon)=\sum_{\lambda\vdash 6} n_g(\lambda) \cdot \chi(U_\varepsilon \cap S_\lambda). 
\]

Since $n_6(\lambda)=1$ and $n_{g \leq 2}(\lambda)=0=n_{g \geq 7}(\lambda)$ 
for all partitions $\lambda\vdash 6$, 
we have $n_6(U_\varepsilon)=1$ and 
$n_{g \leq 2}(U_\varepsilon)=0=n_{g \geq 7}(U_\varepsilon)$. 

For $g=3$, we have 
\begin{align*}
n_3(U_\varepsilon)
&=-1 \cdot \chi(U_\varepsilon \cap S_{(4, 2)})
-1 \cdot \chi(U_\varepsilon \cap S_{(2, 2, 2)}) \\
&=-(128-128)=0. 
\end{align*}

For $g=4$, we have 
\begin{align*}
n_4(U_\varepsilon)
&=3\cdot50+4\cdot128+1\cdot(-216)+4\cdot81+2\cdot(-668) \\
&\quad+3\cdot(-128)+1\cdot968 \\
&=18. 
\end{align*}

For $g=5$, we have 
\begin{align*}
n_5(U_\varepsilon)
&=-4\cdot50-4\cdot128-3\cdot(-216)-4\cdot81-3\cdot(-668) \\
&\quad -2\cdot542-3\cdot(-128)-2\cdot968-1\cdot(-1012) \\
&=-8. 
\end{align*}
\end{proof}

\subsection{GV invariants for twisted Higgs bundles}

\begin{prop}
\label{prop:GVtwHiggs}
We have 
\begin{equation*}
n_{g, 2[C]}(\Tot_C(L))=\begin{cases}
-1 & (g=6), \\
8 & (g=5), \\
-18 & (g=4), \\
8 & (g=3), \\
-2 & (g=2), \\
0 & (\text{otherwise}). 
\end{cases}
\end{equation*}
\end{prop}
\begin{proof}
Let $\hat{M}$ denote the moduli space of semistable $L$-twisted $\PGL_n$-Higgs bundles. 

We first describe a basis of $\mH^i(\hat{M})$ for each $i$. By \cite[Subsection 6.3]{moz12}, adopting the notation of \textit{loc. cit.}, the Poincare polynomial of $\hat{M}$ is
\begin{equation} \label{eq:P(M)}
P_y(\hat{M})=P^{(1)}_{2, 2}(y)/(1+y)^4
=2y^8 + 4y^7 + 8y^6 + 4y^5 + 2y^4 + 4y^3 + y^2 + 1. 
\end{equation}
On the other hand, the cohomology ring $\mH^*(\hat{M})$ has the following  multiplicative generators \cite{dCHM12, ht03, ht04}: 
\[\alpha \in \mH^2(\hat{M}),\quad \beta \in \mH^4(\hat{M}),\quad \psi_i \in \mH^3(\hat{M}), \quad i=1, \cdots, 4. 
\]
Using the relations in \cite[Theorem 1.2.10, Proposition 4.2.1]{dCHM12} and the equation (\ref{eq:P(M)}), we have that 
\begin{align*}
&\mH^0(\hat{M})=\langle 1 \rangle, \quad 
\mH^1(\hat{M})=0, \quad 
\mH^2(\hat{M})=\langle \alpha \rangle, \quad 
\mH^3(\hat{M})=\langle \psi_i \rangle_{1 \leq i \leq 4}, \\
&\mH^4(\hat{M})=\langle \alpha^2, \beta \rangle, \quad 
\mH^5(\hat{M})=\langle \alpha\psi_i \rangle_{1 \leq i \leq 4}, \quad 
\mH^6(\hat{M})=\langle \psi_i\psi_j, \alpha^3, \alpha\beta \rangle_{1 \leq i < j \leq 4}, \\
&\mH^7(\hat{M})=\langle \alpha^2\psi_i \rangle_{1 \leq i \leq 4}, \quad 
\mH^8(\hat{M})=\langle \alpha^4, \alpha\gamma \rangle,  
\end{align*}
where we put $\gamma \coloneqq -2(\psi_1\psi_3+\psi_2\psi_4)$. 
In the following, we only explain the above equalities for $\mH^7(\hat{M})$ and $\mH^8(\hat{M})$. The arguments are similar and simpler for lower cohomological degrees. 
For $\mH^7(\hat{M})$, we have 
\[
\mH^7(\hat{M})=\langle \beta\psi_i, \alpha^2\psi_i  \rangle_{1 \leq i \leq 4}
=\langle \alpha^2\psi_i  \rangle_{1 \leq i \leq 4}, 
\]
where the second equality follows from \cite[Proposition 4.2.1]{dCHM12}. 
On the other hand, by (\ref{eq:P(M)}), we also know that $\dim \mH^7(\hat{M})=4$. Hence, we can conclude that $\{\alpha^2\psi_i\}_{1 \leq i \leq 4}$ is a basis of $\mH^7(\hat{M})$. 
For $\mH^8(\hat{M})$, we have 
\begin{equation} \label{eq:H8}
\begin{aligned}
\mH^8(\hat{M})
&=\langle \alpha^4, \alpha^2\beta, \alpha\psi_i\psi_j, \beta^2 \rangle_{1 \leq i < j \leq 4} \\
&=\langle \alpha^4, \alpha^2\beta, \alpha\psi_1\psi_2, \alpha\psi_1\psi_4, \alpha\psi_2\psi_3, \alpha\psi_3\psi_4, \alpha(\psi_1\psi_3-\psi_2\psi_4), \alpha\gamma, \beta^2 \rangle \\
&=\langle \alpha^4, \alpha\gamma \rangle. 
\end{aligned}
\end{equation}
To see the third equality, we claim that 
\[
\alpha \in I^0_3, \quad \beta^2 \in I^2_1, \quad 
\frac{\alpha^2\beta}{2}+2\alpha\gamma \in I^2_1, 
\]
where $I^a_b \subset \bQ[\alpha, \beta, \gamma]$ is an ideal defined in \cite[Definition 1.2.8]{dCHM12} for $a, b \geq 0$. 
Indeed, we have 
\[
\alpha=\rho^0_{1, 0, 0}, \quad 
\beta^2=\rho^3_{0, 2, 0}, \quad 
\frac{\alpha^2\beta}{2}+2\alpha\gamma =\rho^2_{2, 1, 0}, 
\]
where $\rho^c_{r, s, t} \in I^a_b$ are elements defined in \textit{loc. cit.}. 
Hence, we have that $\beta^2=0$ and $\alpha^2\beta \in \langle \alpha^4, \alpha\gamma \rangle$. 
Furthermore, since we have 
\begin{equation}\label{eq:psis}
\psi_1\psi_2, \psi_1\psi_4, \psi_2\psi_3, \psi_3\psi_4, (\psi_1\psi_3-\psi_2\psi_4) \in \Lambda^2_0, 
\end{equation}
with $\Lambda^2_0$ as defined in \textit{loc. cit.}, we conclude that the elements in (\ref{eq:psis}), multiplied by $\alpha\in I^0_3$, are all zero in $\mH^*(\hat{M})$. 
By the above arguments, the third equality in (\ref{eq:H8}) holds.

Now, we rearrange the above basis of $\mH^*(\hat{M})$ using the perverse degree instead of the cohomological degree. 
Note that the generators $\alpha, \beta, \psi_i$ all have perverse degree $2$ by \cite[Theorem 4.2.2]{dCHM12}. Hence, we obtain
\begin{align*}
&\mH_{\leq0}^*(\hat{M})=\mH_{\leq1}^*(\hat{M})=\langle 1 \rangle, \\
&\mH_{\leq2}^*(\hat{M})=\mH_{\leq3}^*(\hat{M})
=\langle \mH_{\leq 0}^*(\hat{M}),\alpha, \beta, \psi_i \rangle_{1\leq i \leq 4}, \\
&\mH_{\leq4}^*(\hat{M})=\mH_{\leq5}^*(\hat{M})
=\langle \mH_{\leq 2}^*(\hat{M}), 
\alpha^2, \alpha\beta, \alpha\psi_i, \psi_i\psi_j  \rangle_{1 \leq i< j \leq 4}, \\
&\mH_{\leq6}^*(\hat{M})=\mH_{\leq7}^*(\hat{M})=
\langle \mH_{\leq 4}^*(\hat{M}), 
\alpha^3, \alpha^2\psi_i, \alpha\gamma \rangle_{1\leq i \leq4}, \\
&\mH_{\leq8}^*(\hat{M})=
\langle \mH_{\leq 6}^*(\hat{M}), 
\alpha^4 \rangle.
\end{align*}
Here we have followed the notation and normalisation conventions of \cite{dCHM12}, regarding the perverse filtration on $\mH^*(\hat{M})$.  Precisely:
\[
\mH_{\leq i}^*(\hat{M})\coloneqq \mH^{*-a}(\tilde{B}_L,{}^{\mathfrak{p}}\tau_{\leq  i}h_*\bQ[a])
\]
where $a=\dim(\tilde{B}_L)=4g(C)-1$. 

Combined with the isomorphism 
$\mH^*(M) \cong \mH^*(\hat{M}) \otimes \mH^*(\Pic^0(C))$, 
we obtain 
\begin{align*}
\sum_i \chi({}^p\mcH^i({\dR h_* \IC_{M}}))q^i
=-(q^{-4}-2q^{-2}+4-2q^2+q^4)(q^{-1/2}+q^{1/2})^4, 
\end{align*}
from which we can read off the integers 
$n_{g, 2[C]}(\Tot_C(L))$. 
The minus sign in the right hand side comes from the fact that $\dim M=8g(C)-3$ is odd.
\end{proof}

\subsection{GV invariants for $\Tot(N)$}
\begin{thm}
\label{GV_N_thm}
Let $C$ be a genus two curve, and let $L$ be a generic line bundle on $C$ of degree $3$. 
Let $f \in (H^0(L^{\otimes 2}))^\vee \cong \Ext^1(L, L^{-1} \otimes \omega_C)$ 
be a generic linear function, $N$ the corresponding rank two vector bundle. 
Then we have 
\[
n_{g, 2[C]}(\Tot_C(N))=\begin{cases}
8 & (g=3), \\
-2 & (g=2), \\
0 & (\text{otherwise}). 
\end{cases}
\] 
In particular, the GV/GW correspondence holds for $\Tot_C(N)$ and the curve class $2[C]$. 
\end{thm}
\begin{proof}
By Proposition \ref{prop:GV=L+U}, 
we have 
\[
n_{g, 2[C]}(\Tot_C(N))=n_g(U_\varepsilon)
+n_{g, 2[C]}(\Tot_C(L)).
\]
Now the result follows from Propositions \ref{prop:GVnearby} and \ref{prop:GVtwHiggs}.  
\end{proof}

\section{Higher genus} \label{sec:highergenus}
In this section, we fix a smooth projective curve $C$ of genus $g(C)\geq 3$ 
and a line bundle $L$ of degree $2g(C)-1$. 

\begin{prop}
\label{eq:GVLhigh}
We have 
\[
n_{g, 2[C]}(\Tot_C(L))=\begin{cases}
0 & (g \geq 4g(C)-1), \\
-1 & (g=4g(C)-2), \\
-2^{2g(C)-3} & (g=g(C)), \\
0 & (g \leq g(C)-1). 
\end{cases}
\]
\end{prop}
\begin{proof}
For simplicity, we write $M=M_L(2, 1)$ 
and $\hat{M}=\hat{M}_L(2, 1)$. 
By \cite{cl16} (see also \cite[Theorem 0.4]{ms23}), we have 
\[
{}^p\mcH^i(\dR h_* \IC_M) \cong \IC(\wedge^i \dR^1\pi_* \IC_{\mcC}), \quad i \in \bZ, 
\]
where $U \subset \widetilde{B}_L$ is the dense open subset parametrising smooth spectral curves and $\pi \colon \mcC \to U$ is the universal curve. 
Since a smooth spectral curve has genus $4g(C)-2$,
the degree of the Laurent polynomial $\sum_{i \in \bZ} \chi({}^p\mcH^i(\dR h_* \IC_M))q^i$ is less than or equal to $4g(C)-2$. 
The coefficient of $q^{4g(C)-2}$ is 
\[
\chi(\IC_{\widetilde{B_L}})
=(-1)^{\dim \widetilde{B}_L} \cdot \chi(\bQ_{\widetilde{B}_L})
=-1, 
\]
where $\dim \widetilde{B}_L=4g(C)-1$.
This proves that 
\[
n_{g, 2[C]}(\Tot_C(L))=\begin{cases}
0 & (g \geq 4g(C)-1), \\
-1 & (g=4g(C)-2). 
\end{cases}
\]

To determine the GV invariants for small $g$, recall that we have an isomorphism 
\[
\mH^*(M) \cong 
\mH^*(\Pic^0(C)) \otimes \mH^*(\hat{M}). 
\]
Moreover, the generators $\epsilon_i$ 
of $\mH^*(\Pic^0(C))$ have 
cohomological degree one and perverse degree one. Hence we have 
\[
\sum_{i \in \bZ} \chi({}^p\mcH^i(\dR h_* \IC_M))q^i
=(q^{-\frac{1}{2}}+q^{\frac{1}{2}})^{2g(C)}
\cdot \big(\sum_{i \in \bZ} \chi({}^p\mcH^i(\dR \hat{h}_* \IC_{\hat{M}}))q^i \big). 
\]
From this, we can see that 
$n_{g, 2[C]}(\Tot_C(L))=0$ 
for $g \leq g(C)-1$. 
Furthermore, we have 
$n_{g, 2[C]}(\Tot_C(L))=-\chi(\hat{M})$. 
This can be computed as follows: 
Let $P_t(M)$ denote the Poincare polynomial of $M$. 
We can compute $P_{-t}(M)$ by substituting $u=v=-t$ in \cite[Equation (A.7)]{cdp11}: 
\begin{align*}
&\quad P_{-t}(M) \\
&=(1-t)^{2g(C)} \cdot \Bigg\{
\frac{(1-t)^{2g(C)-2}(1+t+t^2)^{2g(C)}}{(1+t)^2(1+t^2)}
+\frac{(1+t)^{2g(C)}}{4(1+t^2)} \\
&\quad\quad\quad\quad\quad\quad\quad\quad+
\frac{t^{4g(C)-2}(1-t)^{2g(C)-2}}{2(1+t)}\left(2g(C)-\frac{1}{1+t} \right) \\
&\quad\quad\quad\quad\quad\quad\quad\quad+
\frac{t^{4g(C)-2}(1-t)^{2g(C)-1}}{2(1+t)}\left(-2g(C)+1-\frac{1}{2} \right) 
\Bigg\}. 
\end{align*}
Substituting $t=1$ in 
$P_{-t}(M)/(1-t)^{2g(C)}$, 
we obtain 
$\chi(\hat{M})=2^{2g(C)-3}$. 
\end{proof}

\begin{prop}
\label{eq:GVUhigh}
The following assertions hold: 
\begin{enumerate}
\item 
$n_g(U_\varepsilon)=\begin{cases}
0 & (g \geq 4g(C)-1), \\
1 & (g=4g(C)-2), \\
0 & (g \leq 2g(C)-2). 
\end{cases}$
\item 
$n_{2g(C)-1}(U_\varepsilon)=
-\sum_{\lambda} \chi(S_{\lambda} \cap U_\varepsilon)
$, where $\lambda=(\lambda_i)$ runs over all partitions of $4g(C)-2$ such that 
$\lambda_i$ is even for every $i$. 
\end{enumerate}
\end{prop}
\begin{proof}
Given a partition $\lambda$ of $4g(C)-2$ 
and an element $a \in S_{\lambda} \cap U_\varepsilon$, 
let $C_a \to C$ be the corresponding spectral curve. 
In view of the recursion formula \eqref{eq:HOMFLYdef} computing 
$P(T_{2, i})$ and the formula \eqref{Hilb_Euler_formula} for $\chi(\Hilb^n(C_a))$, 
we see that 
the polynomial $\int_{C_a^{[*]}}q^{2l}d \chi$ 
has degree 
\[
-2\chi(C)-\deg(L^{\otimes 2}) 
=8g(C)-6 
\]
and the leading coefficient is $1$. 
Hence, we get 
\[
n_g(U_\varepsilon)=\begin{cases}
0 & (g \geq 4g(C)-1), \\
1 & (g=4g(C)-2).
\end{cases}
\]

On the other hand, if we write 
$\int_{C_a^{[*]}}q^{2l}d \chi$ as a polynomial in $(q^{-1}-q)$, 
the lowest term has degree
\begin{equation}
\label{eq:lowestdeg}
-2\chi(C) 
+\# \{i \colon \lambda_i \text{ is odd}\}
\end{equation}
and coefficient $1$. 
Among all the partitions 
$\lambda=(\lambda_i)$, 
(\ref{eq:lowestdeg}) achieves its minimum 
if and only if every $\lambda_i$ is even. 
Hence the remaining assertions also hold. 
\end{proof}

\begin{thm}
Let $L \in \Pic^{2g(C)-1}(C)$ and 
$[N] \in \bP \mH^0(L^{\otimes 2})$ 
be arbitrary. 
Then we have 
\[
n_{g, 2[C]}(\Tot_C(N))=\begin{cases}
0 & (g \geq 4g(C)-2), \\
-2^{2g(C)-3} & (g=g(C)), \\
0 & (g \leq g(C)-1). 
\end{cases}
\]
In particular, the GV/GW correspondence holds for the above range of $g$. 
\end{thm}
\begin{proof}
The proof is the same as that of Theorem \ref{GV_N_thm}, this time combining Proposition \ref{prop:GV=L+U} with Propositions \ref{eq:GVLhigh} and \ref{eq:GVUhigh}.  The final assertion, regarding the GV/GW correspondence, follows by comparison with Corollary \ref{highg_GW}.
\end{proof}

\appendix
\section{Some results from Gromov--Witten theory} \label{sec:GW}
Let $X=\Tot_C(N)$ be a local curve, 
with $N$ a $2$-rigid bundle. 
Consider the following generating series: 
\[
Z_{\leq 2}^{\GW}(X)
=\sum_{g \geq 0} \GW_{g, 1}(X)\lambda^{2g-2}t
+\sum_{g \geq 0} \GW_{g, 2}(X)\lambda^{2g-2}t^2,
\]
where $\GW_{g, d}(X)$ denotes the GW invariants of $X=\Tot_C(N)$ 
for the curve class $d[C]$. 
We \textit{define} $n^{\GW}_{g, d}(X)$ by  the formula (\ref{eq:GVGW}). 
The goal of this appendix is to determine $n^{\GW}_{g, 2}(X)$ for some $g$, 
using the following: 

\begin{thm}[{\cite[Section 8]{bp08}}]
\label{thm:bp}
We have 
\begin{align*}
&\exp\left(
Z^{\GW}_{\leq 2}(X)
\right)
=1+\left(2\sin\left(\frac{\lambda}{2} \right) \right)^{2g(C)-2} \cdot t \\
&+\left(2\sin\left(\frac{\lambda}{2} \right) \right)^{4g(C)-4} \cdot 
\left\{
\left(
4-4\sin\left(\frac{\lambda}{2} \right)\right)^{g(C)-1}+
\left(
4+4\sin\left(\frac{\lambda}{2} \right)\right)^{g(C)-1}
\right\}t^2 \\
&+\cdots.
\end{align*}

\end{thm}

\begin{cor}
Assume that $g(C)=2$. 
Then we have 
\[
n_{g, 1}^{\GW}(X)=\begin{cases}
1 & (g=2), \\
0 & (\text{otherwise}), 
\end{cases} \quad 
n_{g, 2}^{\GW}(X)=\begin{cases}
8 & (g=3), \\
-2 & (g=2), \\
0 & (\text{otherwise}). 
\end{cases} 
\]
\end{cor}
\begin{proof}
It is immediate to read off $n_{g, 1}(X)$ 
from the formula in Theorem \ref{thm:bp}
To compute $n_{g, 2}(X)$, 
it is convenient to work with the variable 
$Q=e^{i\lambda}$. 
Then we have 
\[
\left(Q^{\frac{1}{2}}-Q^{-\frac{1}{2}}\right)^2
=-\left(2\sin\left(\frac{\lambda}{2} \right) \right)^2. 
\]
Firstly, the coefficient of $t^2$ 
in the series $\exp(Z_{\leq 2}^{\GW}(X))$ is 
\begin{align*}
&\quad-\frac{1}{2}\left(Q-Q^{-1}\right)^2
+\frac{1}{2}\left(Q^{\frac{1}{2}}-Q^{-\frac{1}{2}} \right)^4
+\sum_{g\geq 0}(-1)^{g-1}n^{\GW}_{g, 2}(X)\left(Q^{\frac{1}{2}}-Q^{-\frac{1}{2}} \right)^{2g-2} \\ 
&=-2\left(Q^{\frac{1}{2}}-Q^{-\frac{1}{2}} \right)^2
+\sum_{g\geq 0}(-1)^{g-1}n^{\GW}_{g, 2}(X)\left(Q^{\frac{1}{2}}-Q^{-\frac{1}{2}} \right)^{2g-2}. 
\end{align*}
Secondly, the coefficient of $t^2$ 
in the right hand side of the formula in Theorem \ref{thm:bp} is 
\begin{align*}
8\left(Q^{\frac{1}{2}}-Q^{-\frac{1}{2}} \right)^4. 
\end{align*}
Now, we obtain the conclusion by comparing the above two expressions. 
\end{proof}

\begin{cor}
\label{highg_GW}
Let $g(C)$ be arbitrary. 
We have 
\[
n^{\GW}_{g, 1}(X)=\begin{cases}
1 & (g=g(C)), \\
0 & (\text{otherwise}), 
\end{cases} \quad 
n^{\GW}_{g, 2}(X)=\begin{cases}
0 & (g\geq 4g(C)-2), \\
-2^{2g(C)-3} & (g=g(C)), \\
0 & (g \leq g(C)-1). 
\end{cases} 
\]
\end{cor}
\begin{proof}
Similarly to the above corollary, 
we work with the variable $Q=e^{i\lambda}$. 
The coefficient of $t^2$ 
in the series $\exp(Z_{\leq 2}^{\GW}(X))$ is 
\begin{align*}
&-\frac{1}{2}\left(Q-Q^{-1}\right)^{2g(C)-2}
+\frac{1}{2}\left(Q^{\frac{1}{2}}-Q^{-\frac{1}{2}} \right)^{4g(C)-4} \\
&\quad\quad\quad\quad\quad\quad
+\sum_{g\geq 0}(-1)^{g-1}n^{\GW}_{g, 2}(X)\left(Q^{\frac{1}{2}}-Q^{-\frac{1}{2}} \right)^{2g-2}, 
\end{align*}
and 
\begin{align*}
&\quad -\frac{1}{2}\left(Q-Q^{-1}\right)^{2g(C)-2}
+\frac{1}{2}\left(Q^{\frac{1}{2}}-Q^{-\frac{1}{2}} \right)^{4g(C)-4} \\
&=-\frac{1}{2}\left\{ 
\left(Q^{\frac{1}{2}}-Q^{-\frac{1}{2}} \right)^4
+4\left(Q^{\frac{1}{2}}-Q^{-\frac{1}{2}} \right)^2
\right\}^{g(C)-1}
+\frac{1}{2}\left(Q^{\frac{1}{2}}-Q^{-\frac{1}{2}} \right)^{4g(C)-4} \\
&=-2^{2g(C)-3}\left(Q^{\frac{1}{2}}-Q^{-\frac{1}{2}} \right)^{2g(C)-2}+\cdots. 
\end{align*}
Comparing with the right hand side of the formula in Theorem \ref{thm:bp}, 
we obtain the results. 
\end{proof}

\begin{rmk}
Similarly, one can check that the GV type invariants for $\Tot_C(L)$ computed in Propositions \ref{prop:GVtwHiggs} and \ref{eq:GVLhigh} agree with the GW invariants of $\Tot_C(L \oplus (L^{-1}\otimes \omega_C))$ computed 
in \cite[Corollary 7.2]{bp08}. 
\end{rmk}

\bibliographystyle{alpha}
\bibliography{maths}

\end{document}